\documentclass[a4paper,12pt,intlimits,oneside]{amsart}
\usepackage{graphicx}
\usepackage{enumerate}
\usepackage{amsfonts}
\usepackage{amsmath,amsthm,amssymb}
\usepackage{float}
\usepackage{mathabx}
\usepackage{color}

\newcommand{\comment}[1]{}

\newtheorem{theorem}{Theorem}
\newtheorem{definition}{Definition}
\newtheorem{proposition}{Proposition}
\newtheorem{lemma}{Lemma}
\newtheorem{corollary}{Corollary}

\newtheorem{example}{Example}
\newtheorem{remark}{Remark}
\newtheorem{conjecture}{Conjecture}

\newcommand{\ft}[1]{\widehat{#1}}

\newcommand{\con}{\mbox{\rm con\,}}

\newcommand\supp{\qopname\relax o{supp}}

\newcommand\cone{\operatorname{\rm Cone}}

\newcommand{\ffi}{\varphi}
\newcommand\s{\sigma}
\newcommand\si{\sigma}

\newcommand\ve{\varepsilon}

\newcommand\la{\lambda}
\newcommand\de{\delta}
\newcommand\al{\alpha}

\newcommand\om{\omega}
\newcommand\Om{\Omega}

\newcommand{\NN}{\mathbb N}

\newcommand{\RR}{\mathbb R}
\newcommand{\CC}{\mathbb C}
\newcommand{\TT}{\mathbb T}

\newcommand\D{\mathcal{D}}
\newcommand\DD{\mathcal{D}_0}
\newcommand\DC{\mathcal{D}_c}
\newcommand\DDp{\mathcal{D}^{+}}

\newcommand\DL{\mathcal{D}^{\star}}
\newcommand\DN{\mathcal{D}^{\#}}

\newcommand\B{\mathcal{B}}

\newcommand{\FFF}{\mathbb F}

\newcommand\PP{\mathcal{P}}
\newcommand\PPP{\mathcal{P}_0}

\newcommand\BB{\mathcal{B}_0}

\newcommand\CCC{\mathcal C}   %%% folytonos fv.ekre inkább ezt kéne használni talán

\newcommand\MEAN{{\mathbb M}}
\newcommand\MEANG{{\mathbb M}_G}
\newcommand\MEANGH{{\mathbb M}_{\widehat{G}}}

\newcommand\MM{\mathcal{M}}
\newcommand\NNN{\mathcal{N}}

\newcommand\OOO{\mathcal{O}}
\newcommand\NS{\widetilde{\mathcal{N}}}

\newcommand\sat{\s_{\rm at}}

\newcommand\tat{\tau_{\rm at}}

\newcommand\oat{\omega_{\rm at}}

\newcommand{\ler}[1]{\left( #1 \right)}

\newcommand{\ol}[1]{\overline{#1}}

\begin{document}
\title[Integral comparisons on groups]{Integral comparisons of nonnegative positive definite functions on LCA groups}
\author{}
\author{Marcell Ga\'al and Szil\'ard Gy. R\'ev\'esz}

\address{Marcell Ga\'al
\newline  \indent R\'enyi Institute of Mathematics \newline \indent Hungarian Academy of Sciences,
\newline \indent Budapest, Re\'altanoda utca 13-15,  1053 HUNGARY}
\address{and
\newline \indent Bolyai Institute, Interdisciplinary Excellence Centre \newline \indent University of Szeged \newline  \indent  Szeged, Aradi v\'ertan\'uk tere 1, 6720 HUNGARY}
\email{gaal.marcell@renyi.mta.hu}

\address{Szil\'ard Gy. R\'ev\'esz
\newline  \indent R\'enyi Institute of Mathematics \newline \indent Hungarian Academy of Sciences,
\newline \indent Budapest, Re\'altanoda utca 13-15,  1053 HUNGARY} \email{revesz.szilard@renyi.mta.hu}

\date{\today}

\begin{abstract} In this paper we investigate the following questions. Let $\mu, \nu$ be two regular Borel measures of finite total variation. When do we have a constant $C$ satisfying $$\int f d\nu \le C \int f d\mu$$ whenever $f$ is a continuous nonnegative positive definite function?
How the admissible constants $C$ can be characterized, and what is their optimal value?
We first discuss the problem in locally compact abelian groups.
Then we make further specializations when the Borel measures $\mu, \nu$ are both either purely atomic or absolutely continuous with respect to a reference Haar measure. In addition, we prove a duality conjecture posed in our former paper.
\end{abstract}

\maketitle

{\small \tableofcontents}

\bigskip
\bigskip



\newpage
	
\section{Introduction}

In 1988 Logan \cite{Log}, motivated by Montgomery's earlier question, investigated the following extremal problem. For any $T>0$, find the supremum $C(T)$ of the ratio
\[
\frac{\int_{-T}^{T} b(t) dt}{\int_{-1}^{1} a(t) dt}
\]
over the set of all Dirichl\'et polynomials of the form
\[
0 \leq a(t)=\sum_{k=-n}^n a_k e^{\mathrm{i} \lambda_k t}, ~ a_k > 0; \qquad b(t)=\sum_{k=-n}^n b_k e^{\mathrm{i} \lambda_k t}; \qquad |b_k| \leq a_k
\]
where the $\lambda_k$'s ($k=1,\ldots,n$) are real numbers, and $n$ could be arbitrarily large.

By his \emph{mass method}, Logan derived upper bounds for $C(T)$. Furthermore also some lower estimates and continuity properties of the extremal constant (as a function of $T$) have been established.
By taking $b(t)$ to be a translate of $a(t)$, the extremal problem is useful in the estimation of boundary behaviour of
complex analytic functions. If $f$ is an analytic function of the form
\[
f(z)=\sum_{n=0}^{\infty} a_n z^n, \quad a_n \geq 0
\]
whose boundary values belong to $L^2$ on a small arc $$\{ e^{\mathrm{i}t} ~:~ t\in (-\varepsilon,\varepsilon)\}$$ centered at $z=1$, then
according to an observation of Wiener (quoted by Boas \cite{boas})
the $L^2$-norm on the whole boundary is bounded by a constant time the same on the given small arc, which in particular implies that under these assumptions the function $f$ belongs to $L^2$ on the whole unit circle.
Later Erd\H{o}s and Fuchs \cite{EF} proved the explicit inequality
\[
\frac{1}{2\pi} \int_{-\pi}^{\pi} |f(e^{\mathrm{i}t})|^2 dt \leq \frac{3}{2\varepsilon} \int_{-\varepsilon}^{\varepsilon} |f(e^{\mathrm{i}t})|^2 dt, \quad 0 < \varepsilon \leq \pi
\]
improving on the implied constant in Wiener's estimation. For more results on the celebrated Wiener problem the interested reader can consult with the series of publications \cite{BR, GT, Wa, WW}.
Further we mention that Shapiro \cite{Sh} pointed out that to obtain such inequalities for the absolute value squares of positive definite functions, inequalities for nonnegative positive definite functions themselves (without absolute value squares) can be used.

In what follows let us take $f(t):=a(t)=b(t)$. It is a rather trivial observation that taking the measures
\[
d\nu = \chi_{[-T,T]}dx, \qquad d\mu = \chi_{[-1,1]}dx
\]
the above problem can be formulated as follows.
Under what conditions do we have a constant $C$ satisfying
\begin{equation}\label{eq:Halasz}
\int f d\nu \le C \int f d\mu
\end{equation}
whenever $f$ is a nonnegative Dirichl\'et polynomial of positive coefficients?
Note that the occurring functions are positive definite ones, and are also nonnegative.
This motivates the following definition \cite{BAMS}.

\begin{definition}
A function which is both positive definite and nonnegative is called doubly positive.
\end{definition}

So the natural setup is to consider doubly positive functions in this sort of extremal problems.
Therefore, in the present paper we discuss extremal problems in the general setting of LCA (locally compact abelian) groups, taking arbitrary regular Borel measures $\mu, \nu$ of finite total variation and doubly positive functions $f$ vanishing at infinity in place of $f$ in \eqref{eq:Halasz}. The question of finding conditions for \eqref{eq:Halasz} to hold was posed by Hal\'asz in oral communication to us. To the best of our knowledge such comparisons of weighted averages with respect to different weights were first worked out and employed in the seminal number theory paper \cite{Halasz} by Hal\'asz.

Although our results are new in the setting of positive definite functions on the real line as well, we have decided to formulate them for general LCA groups. In this way we could explore those structural properties of LCA groups which indeed play inevitable roles in obtaining the corresponding results, and are somewhat hidden in the particular case of real numbers (or $\mathbb{R}^d$) on which very special algebraic and metric structures are at hand.
We further hope that the choice of the abstract setting may inspire some readers to extend the results say e.g. to  certain noncommutative locally compact groups. In addition, we believe that the techniques developed in the current paper will be useful to address other extremal problems (such as e.g. Delsarte's extremal problem) and their dual descriptions as well.

The first main result of the paper is the following answer to the question of Hal\'asz.

\begin{theorem}\label{th:Halasz} Let $\mu, \nu$ be two (bounded, regular, Borel) real measures on the LCA group $G$. Then \eqref{eq:Halasz} holds for all doubly positive continuous functions $f$ if and only if the measure $C\mu-\nu$ can be decomposed as
\begin{equation*} \label{E:decomp}
C\mu-\nu = \sigma + \tau + o,
\end{equation*}
where $\sigma$ is a nonnegative real measure, $\tau$ is a real measure of positive type and $o$ is an odd measure.
\end{theorem}

This result is somewhat abstract and it is far from easy to check the characterization provided by it. A moment's thought may convince the reader that the condition on the decomposability of the measure $C\mu-\nu$ is a \emph{sufficient} one for inequality \eqref{eq:Halasz} to hold. Yet in the following we hope to convince the reader that the result is neither trivial nor useless.

In \S \ref{sec:atomic} and \S \ref{sec:ac} we intend to make further specializations regarding the occurring measures in Theorem \ref{th:Halasz} in the two most immediate cases. Namely,
using some theory of almost periodic functions, in \S \ref{sec:atomic} we show that in case of atomic measures $\mu,\nu$ in the above decomposition of $C\mu-\nu$ the components can be chosen to be atomic, too. This section owes much to the nice papers of Eberlein \cite{Eberlein1, Eberlein2}.
Next making crucial use of a variant of the Gelfand-Raikov theorem, in \S \ref{sec:ac} we present an absolutely continuous counterpart of the main result of \S \ref{sec:atomic}. This in particular will prove a duality theorem which was conjectured at the end of our former paper \cite{BAMS}.

\medskip

\emph{Acknowledgements.}
The current research was inspired by a comment of G\'abor Hal\'asz on our recent paper \cite{BAMS}.
We also thank him for calling our attention to the references \cite{Halasz, Log, ML}.

This research was partially supported by the DAAD-Tempus PPP Grant 57448965 ,,Harmonic Analysis and Extremal Problems''.

The work of Ga\'al was supported by the National Research, Development and Innovation Office -- NKFIH Reg. No.'s K-115383 and K-128972, and by the Ministry of Human Capacities, Hungary through grant 20391-3/2018/FEKUSTRAT.

The work of R\'ev\'esz was supported by the Hungarian National Research, Development and Innovation Office --  NKFIH Reg. No.'s K-119528 and K-109789.

\section{Harmonic analysis on LCA groups}\label{sec:harmanal}

Let us start with some overview of harmonic analysis on LCA groups.

Let $\mathbb{F}$ be either $\mathbb{C}$ or $\mathbb{R}$.
For a locally compact (Hausdorff) topological space $X$, let us denote by $\mathcal{C}(X,\mathbb{F})$ the set of all continuous $\mathbb{F}$-valued functions on $X$. Furthermore, denote by $\mathcal{C}_c(X,\mathbb{F})$, $\mathcal{C}_0(X,\mathbb{F})$ and $\mathcal{C}_b(X,\mathbb{F})$ the subset of those functions of $\mathcal{C}(X,\mathbb{F})$ which have compact support, vanish at infinity and are bounded, respectively. The inclusions
\begin{equation*}
    \mathcal{C}_c(X,\mathbb{F})\subseteq \mathcal{C}_0(X,\mathbb{F}) \subseteq \mathcal{C}_b(X,\mathbb{F}) \subseteq \mathcal{C}(X,\mathbb{F})
\end{equation*}
are obvious.
If $\mathcal{C}_b(X,\mathbb{F})$ is equipped with the supremum norm $\|.\|_{\infty}$ it forms a normed space $(\mathcal{C}_b(X,\mathbb{F}),\|.\|_{\infty})$. It is well-known that the closure of $\mathcal{C}_c(X,\mathbb{F})$ is $\mathcal{C}_0(X,\mathbb{F})$ in this supremum norm, while the latter is a closed subspace in $\mathcal{C}_b(X,\mathbb{F})$ (and a proper one exactly when $X$ is not compact). Further, $\mathcal{C}(X,\mathbb{F})$ might be equipped with the locally uniform convergence topology $\mathcal{U}_{\rm loc}$. This makes $\mathcal{C}(X,\mathbb{F})$ a topological vector space $(\mathcal{C}(X,\mathbb{F}),\mathcal{U}_{\rm loc})$. The closure of $\mathcal{C}_c(X,\mathbb{F})$ with respect to $\mathcal{U}_{\rm loc}$ is already $\mathcal{C}(X,\mathbb{F})$ -- whence the same holds for the intermediate spaces, too.

\medskip

\textbf{Radon measures.}
Radon measure is one of the most fundamental concepts in abstract harmonic analysis. In fact, there are plenty of different terminologies in the literature. In this paper we adopt the following one.
\begin{itemize}
    \item [a)] A linear functional $T:\mathcal{C}_c(X,\CC)\to \CC$ is termed to be a \textit{(complex) Radon measure} if for any $K\Subset X$ compact set (the symbol $\Subset$ stands for compact inclusion throughout the paper) there exists an $L>0$ such that
$\left|T(f)\right|\leq L\|f\|$ holds whenever $\supp f \subseteq K$.
    \item [b)] A Radon measure $T$ is said to be a \textit{real Radon measure} if for any $f\in \mathcal{C}_c(X,\mathbb{R})$ we have $T(f)\in \mathbb{R}$.
    \item [c)] A Radon measure is said to be a \textit{positive Radon measure} if for any continuous compactly supported function $f\geq 0$ we have $T(f)\geq 0$.
\end{itemize}

    Note that Radon measures in general are assumed to be continuous (or bounded) functionals on neither $(\mathcal{C}_c(X,\mathbb{F}),\|.\|_{\infty})$ nor  $(\mathcal{C}_c(X,\mathbb{F}),\mathcal{U}_{\rm loc})$.
    The family of all $\mathbb{F}$-valued Radon measures will be denoted by $R(X,\FFF)$.
\begin{itemize}
    \item [d)] An element of the dual space $(\mathcal{C}_c(X,\mathbb{F}), \|.\|_{\infty})'=(\mathcal{C}_0(X,\mathbb{F}), \|.\|_{\infty})'$ is called a \textit{bounded Radon measure}.
\end{itemize}

According to the \emph{Riesz representation theorem} (see, for instance, \cite[Appendix E4]{rudin:groups}), there is a one-to-one correspondence between the above abstract notion of Radon measures and the standard measure theoretic one. Namely, the elements of $(\mathcal{C}_c(X,\mathbb{F}), \|.\|_{\infty})'$ are exactly of the form
\begin{equation} \label{ident}
f \mapsto \int_{X} f d\mu \qquad \ler{f \in \mathcal{C}_c(X,\mathbb{F})}
\end{equation}
with some regular complex Borel measure $\mu$ of finite total variation.
Just as in \eqref{ident}, the dual space $(\mathcal{C}_c(X,\mathbb{F}),\mathcal{U}_{\rm loc})'$ can be identified with the set of compactly supported Radon measures \cite[4.10.1 Theorem]{Edwards}.
We introduce the following notations.
\begin{itemize}
    \item[a)] $M(X,\mathbb{F})$: the set of bounded Radon measures, or, equivalently, $\FFF$-valued Borel measures of finite total variation;
    \item[b)] $M_+(X)$: the set of positive bounded Radon measures;
    \item[c)] $M_c(X,\mathbb{F})$: the set of bounded Radon measures with compact support or, equivalently, $\FFF$-valued Borel measures of compact support.
\end{itemize}

Note that the class of general (possibly unbounded) Radon measures can also be characterized in a measure-theoretic way. Let the symbol $\B$ stand for the Borel sigma-algebra with $\BB$ being its subset of Borel sets with compact closure. Then $R(X,\FFF)$ is the family of set functions on $\BB$, the restriction of which to any compact set is a regular Borel measure.

%In the sequel both complex and real valued functions and measures show up frequently.
%To make our notation easier, we drop the notation $\mathbb{C}$ if the corresponding functions or  measures are assumed to be complex valued, and the real valued counterparts are denoted by underlying the same symbols. For instance, we write $\CCC(X), M(X)$ and $\underline{\CCC}(X), \underline{M}(X)$, accordingly.

In what follows let $G$ be an LCA group. For a Borel set $E\subseteq G$, the symbol $\chi_E$ denotes its characteristic (indicator) function.
A Radon measure $\mu$ is said to be translation invariant whenever $\mu(f(x+g))=\mu(f(x))$ holds for any $f\in \CCC_c(G,\mathbb{C})$ and $g\in G$. This condition is equivalent to assuming that the representing measure satisfies $\mu(E+g)=\mu(E)$ for any Borel set $E\in \BB$ and $g\in G$. A nonzero translation invariant Radon measure is called a \textit{Haar measure}. Haar measure exists uniquely up
to a harmless normalization in any LCA group.
A conveniently normalized Haar measure will be denoted by $\lambda$.
As a direct consequence of uniqueness, we also have
$\lambda(E)=\lambda(-E)$ for all Borel measurable set $E$.
%\cite[1.1.4]{rudin:groups}.

\medskip

\textbf{Characters.}
A continuous map $\gamma$ from $G$ into the complex unit circle $\mathbb{T}$
satisfying
\[
\gamma(x+y)=\gamma(x)\gamma(y)\qquad(x,y\in G)
\]
is called a \emph{character} of $G.$ The set of all characters forms a group (with pointwise multiplication) which is called the \emph{dual group} of $G$ and is denoted by $\widehat G$. The dual group $\ft{G}$, equipped with the locally uniform convergence topology, is also an LCA group. Moreover, the Pontryagin-van Kampen duality theorem \cite[(24.8)]{HewittRossII} tells us that the mapping from $G$ to $\ft{\ft{G}}$ defined as
\[
g\mapsto \bigg(\gamma\mapsto \gamma(g), \quad (\forall \gamma \in \ft{G}) \bigg)
\]
provides a homeomorphic group isomorphism between $G$ and $\ft{\ft{G}}$. Hence the characters of $ \widehat{G}$ are exactly the so-called "point value evaluation functionals" at any fixed $g\in G$, by what we mean the functions $\gamma \mapsto \gamma(g)$ ($\gamma \in \ft{G}$).

\medskip

\textbf{Fourier transform.}
For any $\mu \in M(G,\mathbb{C})$, the {\it Fourier-Stieltjes transform} of $\mu$ is denoted by the symbol $\widehat \mu$, that is, we have
$$
\widehat \mu (\gamma)= \int_G \ol \gamma d\mu \quad (\gamma \in \widehat G).
$$
The Fourier-Stieltjes transform of $\ft{\mu}$ is bounded by $\|\mu\|$ and is uniformly continuous on $\ft{G}$ \cite[1.1.3 Theorem]{rudin:groups}.
The \emph{inverse Fourier transform} of a measure $\nu \in M(\ft{G},\mathbb{C})$ is defined by
\begin{equation} \label{invF}
\check{\nu}(x)=\int_{\widehat G} \gamma(x) d\nu(\gamma)\qquad (x \in G).
\end{equation}
The following formula is a version of the so-called \emph{Plancherel theorem} and in fact it is an easy consequence of Fubini's theorem:
if $\nu \in M(G,\mathbb{C})$ and $\sigma \in M(\widehat{G},\mathbb{C})$, then
\begin{equation} \label{eq:parseval}
\nu\left(\overline{{\check{\sigma}}}\right):=\int_G \overline{{\check{\sigma}}} d \nu= \int_G \overline{\int_{\widehat{G}} \gamma(x) d\sigma(\gamma) }d\nu(x) = \int_{\widehat G} \widehat{\nu}  \overline{d\sigma}=: \overline{\sigma}(\widehat{\nu}).
\end{equation}

\medskip

\textbf{Convolution.}
Convolution of two functions $f$ and $g$, or more generally, a function $f$ and a Radon measure $\nu$ could be defined on any LCA group as
\begin{equation} \label{convfunc}
(f\star g)(x) := \int_G f(x-y) g(y)d\lambda(y), \qquad (f\star \nu)(x) := \int_G f(x-y) d\nu(y)
\end{equation}
accordingly, provided that the corresponding integrals exist, for instance, in the following important cases \cite[4.19.2-4]{Edwards}:
\begin{itemize}
    \item [a)] The functions $f$ and $g$ are in $L^1(G,\mathbb{C})$.
    \item [b)] $f \in L^1_{\rm loc}(G,\mathbb{C})$ and $g\in L^1(G,\mathbb{C})$ with compact support.
    \item [c)] $f \in L^1_{\rm loc}(G,\mathbb{C})$ and the Radon measure $\nu$ has compact support.
\end{itemize}

Next we record some useful facts concerning convolution.

\medskip

(R1) If $f \in L^1_{\rm loc}(G,\mathbb{C})$ and $g\in \CCC_c(G,\mathbb{C})$, then $f\star g\in \CCC(G,\mathbb{C})$.

(R2) If $f$ is \emph{locally uniformly integrable} (by which we mean that there is some open set $U$ with compact closure such that $\int_{U+x} |f| \le C $ for all $x\in G$) and $g\in \CCC_c(G,\mathbb{C})$, then $f\star g \in \CCC_b(G,\mathbb{C})$. Further $f\star g$ is uniformly continuous, too.

(R3) The convolution makes the set $L^1(G,\mathbb{C})$ a commutative Banach algebra
possessing the property $\ft{f*g}=\ft{f}\ft{g}$ \cite[Theorem 1.2.4]{rudin:groups} for any $f,g\in L^1(G,\mathbb{C})$.

(R4) The convolution of general Borel functions $f$ and $g$ could be also defined by \eqref{convfunc}, at least in the pointwise sense for those $x\in G$ for which the integral
\[
\int_{G} |f(x-y)g(y)| d\lambda(y)
\]
exists. In the case where $f,g\in L^2(G,\mathbb{C})$ the convolution $f \star g$ is defined everywhere on $G$ and $f \star g \in \CCC_0(G,\mathbb{C})$ \cite[Theorem 1.1.6(d)]{rudin:groups}.

\section{Positive definite functions, integrally positive definite functions, and measures of positive type} \label{Poz}

\medskip

\textbf{Positive definite functions.}
On a LCA group $G$ a function $f$ is called \textit{positive definite} (denoted by $f \gg 0$) if the inequality
\begin{equation}\label{posdefdef}
\sum_{j=1}^{n}\sum_{k=1}^{n} c_j \overline{c_k} f(x_j-x_k)\ge 0
\end{equation}
holds for all choices of $n\in \NN$, $c_j\in \CC$ and $x_j\in G$
for $j=1,\dots,n$.
%First of all, we warn the reader that here (in accordance with for instance \cite{HewittRossI,HewittRossII}) by default the term function means that it takes \emph{finite values everywhere} on its domain. For example, in the paper of Logan \cite{Log} it is not so, as there locally integrable functions come into picture which are defined only almost everywhere (or, equivalently up to equivalence classes modulo a measure zero set).
We shall frequently use the following well-known and immediate consequences of the definition. For any $f\gg 0$, we have \cite[\S 1.4.1]{rudin:groups}
\begin{itemize}
\item[(p1)] $f$ is bounded with $\|f\|_{\infty}\leq f(0)$;
\item[(p2)] the continuity of $f$ all over $G$ is equivalent to that at $0$;
\item[(p3)] $f(x)=\tilde{f}(x):=\overline{f(-x)}$ holds for all $x\in G$
\cite[p.\ 18, Eqn (2)]{rudin:groups}, whence
\item[(p4)] $\supp f$ is
symmetric and the condition $\supp f \subseteq \Om$
implies $\supp f \subseteq \Om\cap(-\Om)$.
\end{itemize}

Let us see some fundamental examples of positive definite functions.

\begin{example}
The characters of a
LCA group $G$ are positive definite. To see this, one needs to use only
the multiplicative property of $\gamma\in \widehat{G}$ to get
$$
\sum_{j=1}^{n}\sum_{k=1}^{n} c_j \overline{c_k} \gamma(x_j-x_k)
= \sum_{j=1}^{n}\sum_{k=1}^{n} c_j\gamma(x_j) \overline{c_k \gamma(x_k)}
= \left|\sum_{j=1}^{n}c_j\gamma(x_j) \right|^2\geq 0
$$
for all choices of $n\in \NN$, $c_j\in \CC$ and $x_j\in G$
for $j=1,\dots,n$.
\end{example}

Together with any $f\gg 0$ also the functions $f^{*}$, where $f^*(x):=f(-x)$, and $\overline{f}$ are positive definite functions.

To obtain further examples, assume that $f$ and $g$ are positive
definite functions on $G$. One checks easily that if $\alpha, \beta>0$
are arbitrary positive constants, then $\alpha f + \beta g \gg 0$.
It is a bit more involved to see that if $f,g\gg 0$, then so is $fg$. This follows from the purely linear algebraic fact that the entrywise product of positive semidefinite matrices is positive semidefinite, as well, which is known as Schur's theorem \cite[\S 85, Theorem 2]{Halmos}.

Taking into account that characters, and also positive linear combinations of characters are positive definite functions, we are led to

\begin{example}
The inverse Fourier transform
\eqref{invF} of any $\nu \in M_+(\ft{G})$ is positive definite.
\end{example}

As it was noted earlier, the inverse Fourier transform of a bounded positive Radon measure is not just positive definite, it is continuous, as well. According to the celebrated
Bochner-Weil-Povzner-Raikov theorem \cite{We,Po,Ra} these latter two properties characterize continuous positive definite functions: a continuous function $f \in \CCC(G,\mathbb{C})$ is positive definite
if and only if it can be obtained as the inverse Fourier transform of
some bounded positive Radon measure $\nu \in M_+({\widehat G})$.

% (see also \cite[10.3.8 Theorem]{Edwards} or \cite[1.4.3. Theorem]{rudin:groups})

Another fundamental and very useful way to get positive definite continuous
functions is taking convolution squares \cite[\S 1.4.2(a)]{rudin:groups}.

\begin{example} \label{E:convsquare}
Let $f\in L^2(G,\mathbb{C})$ be arbitrary. Then the "convolution square" $f\star
\widetilde{f}$ exists and it is a continuous positive definite
function.
\end{example}

This construction of positive definite continuous functions can also be
essentially reversed. The following version \cite[p. 309, (33.24) (a)]{HewittRossII} will be quite useful in the sequel.

\begin{lemma}\label{l:BoasKac}
Let $f\in \D \cap L^1(G)$ be arbitrary. Then there exists a so-called Boas-Kac square-root $g\in L^2(G)$ satisfying $g\star\widetilde{g}=f$. Furthermore, if $f\in \D_c$, then we also have $\supp g \Subset G$.
\end{lemma}

\textbf{Integrally positive definite functions and measures of positive type.}
Beside the above concept of positive definiteness, which is due to Toeplitz \cite{Toeplitz} on $\TT$ and Matthias \cite{Mathias} on $\RR$, there is another notion of positive definiteness meaningful for LCA groups.
A function $f\in L^1_{\rm loc}(G,\mathbb{C})$ is said to be \textit{integrally positive definite} if
\begin{equation} \label{typecond}
\int_{G}f ~(u\star \tilde{u}) ~d\lambda \geq 0 \quad (\forall u\in \CCC_c(G,\mathbb{C})) \quad \textrm{or, equivalently} \quad u \star \widetilde{u} \star f ~(0) \ge 0 \quad (\forall u\in \CCC_c(G,\mathbb{C})).
\end{equation}

We note that a \emph{continuous} integrally positive definite function is necessarily positive definite \cite[Proposition 4]{God}.
Note the distinction between the classes of positive definite functions, defined finitely everywhere and satisfying \eqref{posdefdef}, and integrally positive definite functions, defined only almost everywhere in accordance with \eqref{typecond}.

The advantage of the latter weaker notion of positive definiteness is that it can be carried out for measures.
A Radon measure $\mu$ is said to be a \emph{measure of positive type}\footnote{As a matter of fact, the term 'positive type' is also used for (integrally positive definite) functions but some authors consider this as the synonym of positive definiteness. So we rather use this terminology only for measures.} whenever
\begin{equation}  \label{eq:measurepostyp}
\int_{G} (u\star \tilde{u}) ~d\mu \geq 0 \quad (\forall u\in \CCC_c(G,\mathbb{C})) \quad \textrm{or, equivalently} \quad u \star \widetilde{u} \star \mu ~(0) \ge 0 \quad (\forall u\in \CCC_c(G,\mathbb{C})).
\end{equation}

Recall that if $\mu$ is a measure in $M(G,\mathbb{C})$, then its converse $\tilde{\mu}$ is defined by $\widetilde{\mu}(f):=\overline{\mu}(f^*) = \overline{\mu(\widetilde{f})} $ for all $f\in \CCC_c(G,\mathbb{C})$
which is easily seen to be equivalent to the representing measure satisfying $\widetilde{\mu}(E):=\overline{\mu(-E)}$ for all $E \in \mathcal{B}$.
In analog with (p4) and (p5) concerning positive definite functions, measures of positive type satisfy the following properties.
\begin{itemize}
\item[(p6)] $\widetilde{\mu}=\mu$, whence
\item[(p7)] $\supp \mu$ is symmetric and $\supp \mu \subseteq \Om$
entails $\supp \mu \subseteq \Om\cap(-\Om)$.
\end{itemize}

The property (p6) might be folklore as the analogous statement for positive definite functions. However, we did not find a proper reference for it. Thus, for the sake of completeness we present a self-contained proof of (p6).
We shall invoke the following auxiliary lemma in it.

\begin{lemma} \label{l:measureestimate}
Let $F\subseteq \mathbb{C}$ be a closed set and $\mu \in M(G,\mathbb{C})$ arbitrary. If $\mu(u\star \widetilde{u})\in F$ holds for every $u\in \CCC_c(G,\mathbb{C})$, then we also have $\mu(g\star \widetilde{g})\in F$ for all $g\in L^2(G,\mathbb{C})$.
\end{lemma}

\begin{proof}
Choose an $ \varepsilon> 0 $ and let $g\in L^2(G,\mathbb{C})$ be arbitrary. Note that $g\star \tilde{g} \in \CCC_0(G,\mathbb{C})$ and thus $\mu(g\star \tilde{g})$ exists.
Let us take a function $u\in \CCC_c(G,\mathbb{C})$ such that $\|u-g\|_2 < \varepsilon$.
%(such a $u$ indeed exists, by the density of $\CCC_c(G,\mathbb{C})$ in $L^2(G,\mathbb{C})$).
Observe that $\mu(g\star \widetilde{g})-\mu(u\star \widetilde{u})=\mu((g-u)\star \widetilde{g})+\mu(u\star \widetilde{(g-u)})$ where the terms on the right hand side
can be estimated by Young's inequality as
\[
\left|\mu((g-u)\star \widetilde{g})\right| \leq \|\mu\| \| (g-u)\star \widetilde{g})\|_\infty \leq  \|\mu\| \|g\|_2 \cdot \varepsilon
\]
and similarly
$$
\left|\mu(u\star \widetilde{(g-u)})\right| \leq \|\mu\|
\| u \star (\widetilde{g-u})\|_\infty \leq  \|\mu\| \left( \|g\|_2 + \ve \right)\cdot \varepsilon .
$$
Therefore, for any given function $g\in L^2(G,\mathbb{C})$ the distance
$ \left|\mu(g\star \widetilde{g})-\mu(u\star \widetilde{u})\right|$ can be made arbitrarily small. So any neighborhood of $\mu(g\star \widetilde{g})$ intersects to $F$. Since $F$ is closed, $\mu(g\star \widetilde{g})$ belongs to $F$.
\end{proof}

\begin{proof}[Proof of property (p6)]
Consider the measure $\nu:=\mu-\widetilde{\mu}$ together with its Fourier transform $\ft{\nu}$.
It follows from Lemma~\ref{l:measureestimate} that $\mu (g \star \widetilde{g})\geq 0$ for all $g\in L^2(G,\mathbb{C})$. Take $f:=g \star \widetilde{g}$. In virtue of (R4) $f \in \CCC_0(G,\mathbb{C})$, and clearly $\widetilde{f}=f$. Further we also have $\overline{\mu(f)}=\mu(f)$ because $\mu(f)\geq 0$.
Thus, $$\nu(f)=\mu(f)- \widetilde{\mu}(f)=\mu({f})- \overline{\mu (\widetilde{f})} =\mu(f)- \overline{\mu (f)}=0.$$
Since $\{g\star \widetilde{g} ~ : ~ g\in L^2(G,\mathbb{C})\}$ matches the class $\{\overline{g\star \widetilde{g}} ~ : ~ g\in L^2(G,\mathbb{C})\}$ we infer
\begin{equation} \label{eq:zero}
\int_G \overline{g\star \widetilde{g}} d\nu = 0 \quad (g\in L^2(G,\mathbb{C})).
\end{equation}

The $L^2$-Fourier transform of $g$ being $\widehat g$, we want to show first that the function $g\star \widetilde{g} \in \CCC_0(G,\mathbb{C})$ is the usual $L^1$-inverse Fourier transform of the function $h:=|\widehat g|^2$ in $L^1_{+}(\widehat G,\mathbb{C})$.
Indeed, the $L^2$-Fourier transform is an isometry, whence $\widehat g \in L^2(\widehat G,\mathbb{C})$ and so $h:=|\widehat g|^2 \in L^1_{+}(\widehat G,\mathbb{C})$. Note that $ \widehat{\widetilde{g}} = \overline{\widehat{g}}$, by \cite[(23.10.) Theorem (iv)]{HewittRossI}. Moreover, according to \cite[(31.29) Theorem]{HewittRossII} the inverse Fourier transform of $h=\widehat g \cdot \widehat{\widetilde g} $ (which is identical to the inverse Fourier transform of the nonnegative measure $d\sigma:=h d\lambda_{\widehat G}$), is exactly $g\star \widetilde{g}$.

Now it follows from \eqref{eq:zero} and \eqref{eq:parseval} that
\begin{align*}
0 &= \int_G  \overline{\check{\sigma}} d\nu = \int_{\ft{G}} \widehat{\nu} d \overline{\si} = \int_{\ft{G}} \widehat{\nu} \cdot \overline{h}d \lambda_{\widehat{G}}
\end{align*}
where $\check{\si}= \check{h} = g\star \widetilde{g}, ~ ~ d \sigma = h d\lambda_{\ft{G}}, ~ h=|\widehat g|^2$.
So the integral on $\ft{G}$ of the function $\widehat{\nu} \overline{h}$ vanishes for all $h\in L^1_{+}(\ft{G},\mathbb{C})$, yielding $\widehat{\nu} =0$ first $\lambda_{\ft{G}}$-a.e., and then by continuity everywhere. Therefore, $\nu$ is the zero measure and the proof is complete.
\end{proof}

Although we will not need it here, (p6) can be extended by standard means also to general (i.e. possible unbounded) $\mu \in R(G,\mathbb{C})$.
Let us note that Professor Hal\'asz kindly provided us another, fully general proof relying more on Fourier-Stieltjes transform.

\medskip

\textbf{Positive definiteness and typeness in the real sense.}\label{sec:realposdef}
In the sequel we need a restricted notion of positive and integrally positive definiteness, and positive typeness for the real analysis treatment of duality.
So we introduce the following notions. A real valued function $f$ is \emph{positive definite in the real sense} if it satisfies \eqref{posdefdef} for all real values of the coefficients $c_j$. Analogously, a measure $\mu$ is said to be \emph{a measure of positive type in the real sense} whenever it satisfies \eqref{typecond} for arbitrary real valued weight functions $u\in \CCC_c(G,\mathbb{R})$.

The structural content of positive definiteness originally refers to complex coefficients or weights.
We will see that restricting the conditions to only real coefficients or weights is not equivalent to the restriction of positive definite functions or measures of positive type to assume real values only. More precisely, we have the following statements.

\begin{proposition}\label{pr:realrestriction}
A function $f$ is positive definite and real valued if and only if it is positive definite in the real sense and even.

In addition, a measure $\mu$ is of positive type and real valued if and only if it is a measure of positive type in the real sense and even.
\end{proposition}

\begin{proof}
Since positive definiteness of a function $f:G\to \mathbb{C}$ is characterized by the positive semi-definiteness (in the linear algebraic sense) of the matrix $[f(x_j-x_k)]_{j=1,...,n}^{k=1,...,n}$ for all $n$-tuple $(x_1,\ldots x_n)$,
the first statement follows from the following equivalence:
\[
\langle Az,z \rangle \geq 0 ~ (z\in \mathbb{C}^n)\quad \iff \quad \langle Ay,y \rangle \geq 0 ~ (y\in \mathbb{R}^n), ~ A=A^T
\]
where $A$ is an $n$ by $n$ real matrix and ${}^T$ stands for transposition.

Next we turn to the analogous assertion for measures of positive type.
Assume first that $\mu$ is an even measure of positive type in the real sense, and let $w=u+iv \in \CCC_c(G,\mathbb{C})$ be a complex weight function. For any $u\in \CCC_c(G,\mathbb{C})$ the convolution square $f:=u\star \widetilde{u}$ is positive definite, and obviously it satisfies $\widetilde{f} \equiv f$, too. Thus, we have
$$
\int w\star \widetilde{w} d\mu = \int u\star \widetilde{u} d\mu + \int v\star \widetilde{v} d\mu + i \int \left( -u\star \widetilde{v} + v\star \widetilde{u} \right) d\mu \ge 0,
$$
for the first two integrals need to be nonnegative by assumption, while the function $-u\star \widetilde{v} + v\star \widetilde{u}$ under the last integral sign is odd, hence is orthogonal to the even measure $\mu$.

Conversely, assume now that $\mu$ is a real-valued measure of positive type. Then (p6) entails the evenness of $\mu$, while validity of the integral condition $\int u\star\widetilde{u} d\mu \ge 0$ for real valued weights $u\in \CCC_c(G,\mathbb{R})$ follows from the same, assumed for arbitrary complex weights.
\end{proof}

In accordance with the above, the classes of real functions and measures satisfying the corresponding positive definiteness conditions \eqref{posdefdef} and \eqref{eq:measurepostyp} just for real coefficients or weights are larger than the real valued ones of positive definite functions or measures of positive type, respectively.

\section{The dual cone of the cone of doubly positive continuous functions}\label{sec:dualcone}

In this section our purpose is to characterize the dual cone of doubly positive continuous functions in $M(G):=M(G,\mathbb{R})=(\CCC_0(G,\mathbb{R}),\|.\|_{\infty})'$.
To this end, first let us recall some basic notions and facts concerning dual cones.
Assume that $E$ is a real Banach space with dual space $E'$. If $K\subseteq E$ is a set, then the cone generated by $K$ will be denoted by $\cone(K)$. For any set $S \subseteq E$, the dual cone of $S$ is denoted by $S^{+}$ and defined as
$$
S^{+}=\left\{ \varphi \in E' \, : \, \varphi(x) \geq 0 \quad (\forall x \in S) \right\}.
$$
Note that $S^{+}$ is always a closed cone.
Let us introduce the following notation.
\begin{itemize}
    \item[a)] $\PP$: the cone of nonnegative continuous functions;
    \item[b)] $\PP_c$, $\PPP$, $\PP_{\infty}$: the nonnegative cones of $\CCC_{c}(G,\mathbb{R})$ and $\CCC_{0}(G,\mathbb{R})$, respectively, that is, $$\PP_c = \mathcal{P} \cap \CCC_{c}(G,\mathbb{R}), \qquad  \PPP=\mathcal{P} \cap \CCC_{0}(G,\mathbb{R}), \qquad \PP_{\infty}=\PP \cap \CCC_b(G,\mathbb{R}); $$
    \item[c)] $\mathcal{D}$: the family of all continuous real valued positive definite functions;
    \item[d)]  $\mathcal{D}_c, \mathcal{D}_0$: the cones of positive definite elements in $\CCC_{c}(G,\mathbb{R}),~ \CCC_0(G,\mathbb{R})$, respectively, that is,
\begin{equation*} \label{eq:ddef}
\mathcal{D}_c=\D \cap \CCC_{c}(G,\mathbb{R}), \qquad \DD=\D \cap \CCC_{0}(G,\mathbb{R});
\end{equation*}
   \item[e)] $\DL$: the set of real valued integrally positive definite functions.
\end{itemize}

A convex cone is subject to the appropriate version of the Krein-Milman or Choquet theorem, so its closure consists of limits of linear combinations of its extreme points. However, describing the extreme points of e.g. $\PPP\cap\mathcal{D}_0$ is still an open problem even for the most immediate case of $G=\RR$. For more on this problem, attributed to Choquet, the interested reader can consult with the publication \cite{JMR}.

In the lights of the above, one may expect that the dual cone of $\PPP\cap\mathcal{D}_0$ is even less easy to describe. Still, we will find a straightforward description.
Our major tool to the solution of this problem will be an intersection formula on the dual cones of the intersection of two cones. Basically, what we are after is an intersection formula stating $(A\cap B)^+=A^+ + B^+$, the point being that from the generally true statement $(A\cap B)^+=\overline{A^+ +  B^+}$, where the closure is taken with respect to the weak-star topology, we would like to get rid of the closure.

Such theorems are known to hold particularly when the intersection of the cones is large enough or one of the cones has a nonempty interior. However, the cones we are considering here will not provide us such easy criteria for an intersection formula to hold, as neither $\PPP$ nor $\mathcal{D}_0$ has a nonempty interior in $(\CCC_0(G,\mathbb{R}),\|\cdot\|_\infty)$. Indeed, for any $f\in \PPP$ the $\varepsilon$-neighbourhood contains ultimately negative functions because $f$ vanishes at infinity. Furthermore, it is even more obvious that for any $f\in  \mathcal{D}_0$ one finds $\CCC_0(G,\mathbb{R})$ functions arbitrarily close to it in $\|\cdot\|_\infty$ but not admitting the property (p4) of positive definite functions. The same way, (p2) fails for arbitrarily close functions, too.
To circumvent the difficulty, we invoke a lesser known version of the intersection formula which gives the same conclusion under different hypothesis. The precise formulation of the corresponding statement (see \cite[Lemma 2.2.]{jeya-wolk} or \cite[Section 15.D]{GeomAnal}) reads as follows.

\begin{lemma}\label{l:Jeya-Wolk} Assume that $A$ and $B$ are closed convex sets in a real Banach space $E$. If ${ \bf 0} \in A \cap B$ and $\cone(B-A)$ is a closed subspace of $E$, then in $E'$ we have
\begin{equation} \label{eq:insecform}
\ler{A \cap B}^{+}=A^{+} + B^{+}.
\end{equation}
\end{lemma}

It seems highly non-trivial that the conditions of Lemma~\ref{l:Jeya-Wolk} hold when we consider the real Banach space $E={\CCC_0}(G,\mathbb{R})$ and its subsets $A=\PPP$, $B=\mathcal{D}_0$. This is our point with the next lemma.

\begin{lemma}\label{l:D-P} For any $f\in {\CCC_0}(G,\mathbb{R})$, there exists $F \in \mathcal{D}_0$ such that $f\le F$. In other words, $\mathcal{D}_0-\PPP={\CCC_0}(G,\mathbb{R})$.
\end{lemma}

The proof rests heavily on the following lemma concerning the locally uniform approximation of the constant one function.

\begin{lemma}\label{l:Kg} Let $K\Subset G$ be arbitrary. Then for any $\varepsilon>0$ there exists $g\gg 0$, $g\in \CCC_{c}(G,\mathbb{R})$ with  $g\ge 0$, $g|_K \ge 1$ and $\|g\|_\infty=g(0)\le 1+\varepsilon$.
\end{lemma}

For details see, for instance, \cite[Problem 5]{God} or \cite[2.6.8. Theorem]{rudin:groups}.

\begin{proof}[Proof of Lemma~\ref{l:D-P}] Assume, as we may, $\|f\|_\infty=1$. Define the subsets $$K_n:=%\overline{
\{x\in G~:~ |f(x)|>2^{-n}\}.$$ Therefore, $K_0=\emptyset$ and all the sets $K_n$ are relatively compact in view of "$\lim_\infty f = 0$".
Thus, $K_n$ ($n\in \mathbb{N}$) is an increasing sequence from $\BB$ with
\[
H_n:=K_n \setminus K_{n-1}  = \{x\in G~:~ 2^{-n} < |f(x)|\le 2 \cdot 2^{-n} \}.
\]
Consequently, on $H_n$ we have $f \le 2^{1-n} g_n$ with the function $g_n$ constructed for $\ve=1$ and the set $K_n$ by means of Lemma~\ref{l:Kg}. Note that
\[
\bigcup_{n=1}^\infty H_n = \bigcup_{n=1}^\infty K_n = \{ x\in G~:~ f(x)\ne 0\}.
\]
Recall that we also have $g_n\ge 0$. Therefore,
$$f \le F:=\sum_{n=1}^\infty 2^{1-n}g_n$$ on the whole $G$ where the series converges normally and thus uniformly. So we find that $F\in {\CCC_0}(G,\mathbb{R})$. Moreover, $F\in \mathcal{D}_0$ holds, too.
\end{proof}

\begin{corollary}\label{l:intersection} We have $(\PPP \cap \mathcal{D}_0)^{+}= \PPP^{+} + \mathcal{D}_0^{+}$ in ${M}(G)$.
\end{corollary}

\begin{proof} Clearly, $\PPP$ and $\DD$ are closed cones in ${\CCC_0}(G,\mathbb{R})$ and ${\bf 0} \in \DD \cap \PPP$. Furthermore, we have learnt in Lemma \ref{l:D-P} that $\DD -\PPP={\CCC_0}(G,\mathbb{R})$ which is obviously a closed subspace of itself. Therefore, Lemma~\ref{l:Jeya-Wolk} applies.
\end{proof}

In the remaining part of the section, we give a precise description of the dual cones appearing in Corollary~\ref{l:intersection}.

\begin{theorem}\label{th:dualcone} The dual of the cone of real valued continuous nonnegative positive definite functions vanishing at infinity is the Minkowski sum of the cone of nonnegative measures, the cone of measures of positive type $\MM$ in $M(G)$ and the family of odd measures $\mathcal{O}$ in $M(G)$. That is, we have
$$
(\PPP\cap\DD )^+=M_+(G) + \MM + \mathcal{O}.
$$
\end{theorem}

For the proof we need the following characterization of measures of positive type.

\begin{proposition}\label{P:newBochner}
For a bounded Radon measure $\mu$, the following statements are equivalent.
\begin{itemize}
    \item[(i)] $\mu$ is of positive type;
    \item[(ii)] $\mu(f)\geq 0$ holds for all continuous positive definite $f \gg 0$.
    %\item[(iii)] the Fourier-Stieltjes transform $\gamma \mapsto \widehat{\mu}(\gamma)$ is nonnegative.
\end{itemize}
\end{proposition}

\begin{proof}
The implication (ii)$\implies$(i) is immediate, so we verify (i)$\implies$(ii).
To see this, note that according to \cite[p. 309, (33.24) (a)]{HewittRossII} any positive definite continuous $L^1(G,\mathbb{C})$-function $f$ arises as a convolution square $f=g\star \widetilde{g}$ with some $g\in L^2(G,\mathbb{C})$. In virtue of Lemma~\ref{l:measureestimate} this means that (i) implies $\mu(f)\geq 0$ for all positive definite continuous $L^1(G,\mathbb{C})$-function $f$. Next we show that this implies (ii).

Let $f$ be arbitrary continuous positive definite function.
Assume, as we may, $\|f\|_\infty=1$.
Let now $\de>0$ be fixed and $K\Subset G$ be a compact set such that $|\mu|(G\setminus
K)<\de$. Furthermore, for the given value of $\de>0$ and for the above defined compact set $K\Subset G$, let $k$ be an approximation of the identically 1 function ${\bf 1}$, provided by Lemma \ref{l:Kg}. Apparently, we have
$$
\int f d\mu = \int_{G\setminus K} (1-k)f d\mu +
\int_K (1-k) f d\mu + \int_G k f d\mu
$$
where, using $\|f\|_{\infty}=1$, the first two terms can be estimated as
$$
\Bigl|\int_{G\setminus K}
(1-k)f d\mu \Bigr| \le |\mu|(G\setminus K) < \de, \qquad
\Bigl|\int_K (1-k) f d\mu \Bigr| \le  \|\mu\| \de.
$$
It follows that for any given $\ve>0$ and then suitable $\de>0$, $K\Subset G$ and $k \approx {\bf 1}$, it holds
$$
\Bigl| \int_{G} f d\mu - \int_{G} k f d\mu \Bigr| <\ve.
$$
Since the positive definite continuous function $kf$ belongs to $L^1(G,\mathbb{C})$, we infer $\int_{G} k f d\mu \geq 0$. Therefore, from the last displayed inequality we conclude that for any $\ve>0$ the
distance of $\int f d\mu$ from $[0,+\infty)$ is $<\ve$, whence $\int f d\mu \ge 0$, as wanted.

% As for (iii)$\implies$(ii), let $\mu$ be a bounded Radon measure with nonnegative Fourier-Stieltjes transform.
%Take any positive definite continuous function $f$ and represent it as $f=\check{\sigma}$ with some $\sigma \in M(\widehat{G})_+$ provided by Bochner's theorem. Then referring to \eqref{eq:parseval} we have   $\mu(\overline{f})=\mu(\overline{\check{\sigma}})=\overline{\sigma}(\widehat{\mu})=\sigma(\widehat{\mu})\geq 0$.
%Since the condition $\mu(f)\geq 0$ for all continuous $f\gg 0$ is equivalent to that of $\mu(\overline{f})\geq 0$ for all continuous $f\gg 0$, the proof is complete.

\end{proof}

\begin{proof}[Proof of Theorem~\ref{th:dualcone}]
According to Corollary~\ref{l:intersection}, it suffices to characterize the dual cones of the sets $\PPP$ and $\DD$ themselves. For the first, $\mathcal{P}_0^+=M_{+}(G)$ is well-known.
As for the second, it follows from Proposition~\ref{P:newBochner} that $\mathcal{M} \subseteq \DDp \subseteq \DD^+ $. Moreover, if $o$ is an odd measure, then for any real valued continuous positive definite function $f$ the integral $o(f)=\int_G f d o$ vanishes (for $f$ is even, too). So we find $\OOO \subseteq \DD^+$ which yields $\MM+\OOO \subseteq \DD^+$.

To prove the converse, assume that $\mu$ lies in $\DD^+$. As $\CCC_c(G,\mathbb{R})\star \CCC_c(G,\mathbb{R}) \subseteq \CCC_0(G,\mathbb{R})$, we see that $\mu$ is a measure of positive type in the real sense.
Observe that for any Radon measure, we always have for arbitrary weights $u\in \CCC_c(G,\mathbb{R})$ that $$\widetilde{\mu}(u\star \widetilde{u}) = \int (u\star \widetilde{u}) d\overline{\mu^*} = \int (u\star \widetilde{u})^* \overline{ d\mu} = \overline{\int \widetilde{(u\star \widetilde{u})} d\mu }= \overline{\int (u\star \widetilde{u}) d \mu} = \overline{\mu(u\star \widetilde{u})}$$ for $\widetilde{(u\star \widetilde{u})}=u\star \widetilde{u}$. Therefore, $\widetilde{\mu}$ is a measure of positive type in the real sense, too.
It means that the even part $\nu$ of $\mu$, that is, $\nu:=(\mu+\widetilde{\mu})/2$ is a measure of positive type and thus $\nu \in \mathcal{M}$.

Now the desired characterization follows rather easily. Indeed, we just need to note that if $\mu\in M(G)$, then $\mu$ can be decomposed uniquely as $\mu=\nu+o$ with $o:=(\mu-\widetilde{\mu})/2 \in \OOO$. The proof is complete.
\end{proof}

The following Bochner type characterization of measures of positive type will also be useful.

\begin{proposition}\label{P:comBochner}
For a bounded Radon measure $\mu$, the following statements are equivalent.
\begin{itemize}
    \item[(i)] $\mu$ is a measure of positive type;
    \item[(ii)] the Fourier-Stieltjes transform $\widehat{\mu}$ is nonnegative.
\end{itemize}
\end{proposition}

\begin{proof}
Let $\gamma \in \widehat{G}$. Since the characters are continuous positive definite functions, we obtain from Proposition~\ref{P:newBochner} that $\widehat{\mu}(\gamma)=\mu(\overline{\gamma}) \geq 0$. Thus (i) implies (ii).

To see the converse, note that the assumption $\mu(f)\geq 0$ for all continuous positive definite $f$ is clearly equivalent to assuming $\mu(\overline{f})\geq 0$ for all $f$ from the same class. So represent any continuous positive definite function as
$f=\widecheck{\sigma}$ with $\sigma \in M_{+}(\widehat{G})$ being the occurring measure in the Bochner-Weil-Povzner-Raikov theorem.
Then according to \eqref{eq:parseval} we have $\mu(\overline{f})=\mu(\overline{\widecheck{\sigma}})=\overline{\sigma}(\widehat{\mu})={\sigma}(\widehat{\mu})\geq 0$ as wanted.
\end{proof}

\begin{corollary}\label{C:realsense}
A bounded Radon measure is a measure of positive type in the real sense if and only if its Fourier transform has nonnegative real part.
\end{corollary}

\begin{proof}[Proof of Corollary~\ref{C:realsense}]
For temporary use, let us introduce the following sets.
\begin{itemize}
    \item[a)] $ \NNN:= \{\mu \in M(G) ~:~ \mu\left( u \star \widetilde{u}\right) \ge 0 ~ \mbox{for all } u\in {\CCC_c}(G,\mathbb{R}) \}$
    \item[b)]
        $ \NNN^2:= \{\mu \in M(G) ~:~ \mu\left( u \star \widetilde{u} \right)  \ge 0 ~ \mbox{for all } u\in {L^2}(G,\mathbb{R}) \}$
    \item[c)] $\NNN^{\wedge}_{+}:=\{ \mu \in M(G) ~:~ \Re \widehat{\mu}(\gamma) \ge 0  \}$
    \item[d)] $\widetilde{\mathcal{N}}:=\{ \mu \in M(G) ~ : ~ \mu+\widetilde{\mu} \in \mathcal{M} \}$
    \end{itemize}
We establish the following equivalences.
\begin{equation} \label{Nnek}
    \DD^+ = \NNN = \NNN^2 = \NNN^{\wedge}_{+} = \widetilde{\mathcal{N}}
\end{equation}
It is not difficult to check that $\widetilde{\mathcal{N}}=\mathcal{M}+\mathcal{O}$.

Using Proposition~\ref{P:newBochner} it follows that $\MM \subseteq \D^+$. Further if $o$ is an odd measure, then for any even $f\in {\CCC}_b(G,\mathbb{R})$, and so for any real valued continuous positive definite function $f$ the integral $o(f)=\int_G f d o$ vanishes. Thus, we find $\MM,\OOO \subseteq \D^+$ which furnishes $\MM+\OOO \subseteq \D^+$.

The inclusions ${\CCC_c}(G,\mathbb{R})\subseteq {L^2}(G,\mathbb{R})$ and ${L^2}(G,\mathbb{R})\star {L^2}(G,\mathbb{R}) \subseteq {\CCC_0}(G,\mathbb{R})\subseteq {\CCC}(G,\mathbb{R})$ entail that $\NNN\supseteq \NNN^2\supseteq\D_0^{+}\supseteq\D^{+}$.
So up to here we have seen that $$\MM+\OOO =\NS \subseteq \D^{+} \subseteq \D_0^{+} \subseteq \NNN^2 \subseteq \NNN.$$ Therefore, it remains to verify $\NNN \subseteq \NNN^{\wedge}_{+} \subseteq \NS$.

Now consider the inclusion $\NNN^{\wedge}_{+} \subseteq \NS$. For any real measure $\mu\in {M}(G)$ an easy calculation yields for any $\gamma \in\ft{G}$ that (see also p.16 in \cite{rudin:groups}) $$\overline{\widehat{\mu}(\gamma)} = \overline{\int_G \overline{\gamma} d\mu} = \int_G \gamma d\mu = \int_G \gamma^{*} d\mu^{*} = \int_G \overline{\gamma} d\widetilde{\mu} = \widehat{\widetilde{\mu}}(\gamma).$$
Therefore, $2 \Re \widehat{\mu}(\gamma)= \widehat{\mu}(\gamma)+\overline{\widehat{\mu}(\gamma)} = \widehat{\mu+\widetilde{\mu}}(\gamma)$, that is, $\Re \widehat{\mu} = \widehat{\nu}$ with $\nu$ being the even part of $\mu$. So if $\mu \in \NNN^{\wedge}_{+}$, then we find that the even component $\nu$ satisfies $\widehat{\nu} = \Re \widehat{\mu} \ge 0$, that is, $\nu\in M^{\wedge}_+(G)$. According to Proposition~\ref{P:comBochner} we have, however, $M^{\wedge}_+(G) = \MM$, so that $\nu \in \MM$. This proves $\mu \in \NS$ and thus $\NNN^{\wedge}_{+} \subseteq \NS$ as wanted.

To conclude the proof, we prove the inclusion $\NNN \subseteq \NNN^{\wedge}_{+}$. Note that for any Radon measure, we always have for arbitrary weights $u\in \CCC_c(G,\mathbb{R})$ that $$\widetilde{\mu}(u\star \widetilde{u}): = \int (u\star \widetilde{u})^* \overline{ d\mu} = \overline{\int \widetilde{(u\star \widetilde{u})} d\mu }= \overline{\int (u\star \widetilde{u}) d \mu} = \overline{\mu(u\star \widetilde{u})}$$ for $\widetilde{(u\star \widetilde{u})}=u\star \widetilde{u}$, always.
So in particular for any $\mu \in \NNN$ also $\widetilde{\mu} \in \NNN$. Therefore, even $\nu:=(\mu+\widetilde{\mu})/2 \in \NNN$. However, $\nu$ is also even, whence by the second part of Proposition \ref{pr:realrestriction} we find $\nu \in \MM$. Thus, in view of Proposition~\ref{P:comBochner} $\nu \in M^{\wedge}_{+}(G)$ as well. That is, we get $\widehat{\nu} \ge 0$. However, $\widehat{\nu} = \Re \widehat{\mu}$ for any $\mu\in {M}(G)$, whence this implies $\mu \in \NNN^{\wedge}_{+}$ concluding the proof.
\end{proof}

\begin{remark}\label{R:MNO}
Observe that in the course of Proof of Theorem~\ref{th:dualcone} we have established $\mathcal{D}_0^+=\mathcal{M}+\mathcal{O}$ while here we have seen $\mathcal{N}=\mathcal{M}+\mathcal{O}$.
\end{remark}

\section{Shapiro-type extremal problems on LCA groups}\label{sec:preparations}

In Section~\ref{Poz} we have introduced various function spaces and corresponding notions of positive definite functions or integrally positive definite functions. In what follows we define certain extremal quantities on these function classes.

\begin{definition}\label{def:Qdef} Let $U\in \BB$ be a symmetric neighborhood of $0$ and $k\in \NN$. Denote $kU:=\underbrace{U+U+\dots+U}_\text{$k$ {\rm times}}$. We define the extremal quantities
\begin{equation}\label{eq:Qdef}
Q(U,k):=\sup_{\substack{0 \le f \in \D \\ f \not\equiv 0 ~ {\rm loc ~ a.e.}}}  \frac{\int_{kU} f d\lambda}{\int_U f d\la};   \quad
Q_c(U,k):=\sup_{\substack{0 \le f \in \DC \\ f \not\equiv 0 ~ {\rm loc ~ a.e.}}}  \frac{\int_{kU} f d\lambda}{\int_U f d\la};   \quad
Q^\star(U,k):=\sup_{\substack{0 \le f \in \DL \\ f \not\equiv 0 ~ {\rm loc ~ a.e.}}}  \frac{\int_{kU} f d\lambda}{\int_U f d\la}.
\end{equation}
\end{definition}

In case of $\RR$ or even $\TT$ (with, for example, some small enough $\de$ and $U=]-\de,\de[$), analogous questions can be posed for any dilate $\kappa U$ of $U$ with any $\kappa>0$. However, in general LCA groups there may not exist dilates. So considering $kU$ is the next best thing to imitate the nature of the problems for $\RR$ and $\TT$. Dilates are considered for the Wiener problem even on $\RR^d$ and $\TT^d$ in \cite{GT} and the most important base sets occurring there are balls and cubes, which are centrally symmetric convex bodies, so in these cases $kU$ in the dilate sense equals to the above defined $kU$ in the multiple self-addition sense. In some other cases (for example, in cases of starlike domains) there is an asymptotic equivalence with $kU\sim k (\con U)$ in higher dimensions\footnote{Recall that according to Caratheodory's theorem in dimension $d$ for every convex combinations $v\in V:=\con U$ there exists $d+1$ elements of $U$ and nonnegative coefficients $\al_j\geq 0$ summing to 1 such that $\sum_{j=1}^{d+1} \al_j u_j =v$. Then it is easy to show that $kU^\star\subseteq kV \subseteq (k+d)U^\star$ with $U^\star:= \{\al u~:~ 0\le \al\le 1, u\in U\}$ being the "starlike hull" of $U$. It means that for large $k$ the deviation remains bounded.}. However, in general LCA groups we cannot even define $\con U$, whence there is no obvious way to compare our definition with the one in \cite{GT} analyzed in $\RR^d$ and $\TT^d$.

\begin{definition}\label{def:Sdef}
Let $U,V \in \BB$ with $U$ a symmetric neighborhood of $0$. We define the extremal quantities
\begin{equation}\label{eq:Sdef}
S(U,V):=\sup_{\substack{0 \le f \in \D \\ f \not\equiv 0 ~ {\rm loc ~ a.e.}}}  \frac{\int_{V} f d\lambda}{\int_U f d\la};   \quad
S_c(U,V):=\sup_{\substack{0 \le f \in \DC \\ f \not\equiv 0 ~ {\rm loc ~ a.e.}}}  \frac{\int_{V} f d\lambda}{\int_U f d\la};   \quad
S^\star(U,V):=\sup_{\substack{0 \le f \in \DL \\ f \not\equiv 0 ~ {\rm loc ~ a.e.}}}  \frac{\int_{V} f d\lambda}{\int_U f d\la}.
\end{equation}
\end{definition}

In particular, $Q(U,k):=S(U,kU)$ etc.

\begin{definition}\label{def:Tdef}
Let $U\in \BB$ be a symmetric neighborhood of $0$ and $g\in G$ be arbitrary. We define the extremal quantities
\begin{equation}\label{eq:Tdef}
T(U,g):=\sup_{\substack{0 \le f \in \D \\ f \not\equiv 0 ~ {\rm loc ~ a.e.}}}  \frac{\int_{U+g} f d\lambda}{\int_U f d\la};   \quad
T_c(U,g):=\sup_{\substack{0 \le f \in \DC \\ f \not\equiv 0 ~ {\rm loc ~ a.e.}}}  \frac{\int_{U+g} f d\lambda}{\int_U f d\la};   \quad
T^\star(U,g):=\sup_{\substack{0 \le f \in \DL \\ f \not\equiv 0 ~ {\rm loc ~ a.e.}}}  \frac{\int_{U+g} f d\lambda}{\int_U f d\la}.
\end{equation}
\end{definition}

That is, $T(U,g)=S(U,g+U)$ etc.

These might be called the Shapiro-type extremal problems on $\D$, $\DC$ and $\DL$, accordingly. If we replace the nonnegativity assumption $ f\geq 0$ on the integrands by simply considering arbitrary $f\gg 0$ but writing $|f|^2$ in the integrands, then we obtain the Wiener type extremal quantities. Any Shapiro-type extremal quantity is at least as large as the corresponding Wiener-type extremal quantity because for any $f\gg 0$ the function $|f|^2$ is always a doubly positive function.

Several variants of these extremal quantities appear in the literature. A basic one is the analogous quantity with $\DD$ replacing $\D$ -- we may write $Q_0(U,k), S_0(U,V), T_0(U,g)$ for the arising extremal quantities.

If the conditions are $f(x)=\sum_{j=1}^m a_j \Re \gamma_j(x)$ with $m\in \NN$,
$\gamma_j \in \ft{G}$ and $a_j\ge 0$, that is, if the class is the family of the
trigonometric polynomials from $\D$, then we may term the
corresponding special case the \emph{second Montgomery-Logan-Shapiro type
question}\footnote{The first problem is the analogous question with
Hardy-Littlewood type majorization. We do not deal with this type of
more general question here.}. The Montgomery-type question was
originally posed for Dirichlet polynomials on $\RR$. In fact, according
to the Bochner-Weil-Povzner-Raikov theorem any $f \in \D$ is the inverse Fourier transform of a nonnegative Radon measure $\nu \in M_+(\ft{G})$, and the
Montgomery-type case corresponds to the restriction of occurring
measures to nonnegative measures having finite support: $$\nu \in
M^\#_+(\ft{G}):=M_+(\ft{G}) \cap M^\#(\ft{G})$$ where
$M^\#(\ft{G}):=\{\nu \in M(\ft{G})~:~ \# \supp \nu  <\infty\}$. So we may set
$$\DN:=\{ f\in C(G)~:~ f=\check{\nu}, \nu \in  M^\#_+(\ft{G})\}.$$
Accordingly, we may denote the corresponding extremal problems
(extended over functions $0\le f \in \DN$) by $Q^\#(U,k), S^\#(U,V)$
and $T^\#(U,g)$, respectively. If again here we consider integrals of
$|f|^2$ for all $f\in \DN$, then we obtain the original "second
Montgomery-type question" -- it was posed for the comparison of the
integrals (over a centered interval on $\RR$ versus over some arbitrary
interval) of the absolute value square of a Dirichlet polynomial having
nonnegative coefficients.

Further, it is possible here to involve other restrictions on the representing measures occurring in the Bochner type representation of the continuous positive definite functions $f \in \D$ which we consider. If we take all atomic measures $\nu$ (of not necessarily finite support), we obtain the class ${\D^{\rm a}}$, and the respective problems $Q^{\rm a}(U,k), S^{\rm a}(U,V)$ and $T^{\rm a}(U,g)$; and we can as well restrict considerations to Rajchman measures (getting ${\D^{\rm R}}$) and absolutely continuous measures (getting ${\D^{\rm ac}}$). Although these (and other variants) do occur in the literature, we do not pursue these variants any further because of a good reason explained below.

\begin{theorem}\label{prop:AllS} Let $U,V \in \BB$ with $U$ a symmetric neighborhood of $0$. Then $S_c(U,V)=S^*(U,V)$.
\end{theorem}

\begin{corollary} All these extremal problems are equivalent, that is, we have $ S_c(U,V)=S_0(U,V)=S(U,V)=S_1(U,V)=S^\star(U,V)=S^{\rm a}(U,V)=S^{\rm R}(U,V)=S^{\#}(U,V)=S^{\rm ac}(U,V)$.

The same identifications hold for the extremal quantities $Q$ and $T$, too.
\end{corollary}

\begin{remark} This is interesting for Logan \cite{Log}, when constructing lower estimates for the extremal constants $S^{\star}([-1,1],[-T,T|)$ on $\RR$  emphasizes the role of unfamiliar behavior of singular positive definite functions (i.e. $f\in \DL$), as opposed to familiar behavior of continuous positive definite functions (i.e. $f\in \D$). Now it seems that even if finding lower estimates may be simpler in the wider class of $\DL$ -- e.g. manipulations with special functions obtained by singular series becomes possible -- but there can be no difference regarding the extremal values themselves.
\end{remark}

\begin{remark}
The question of deciding whether the value of the extremal constant varies by passing to a different function class is often a rather challenging problem (if doable at all). As for recent investigation in this direction, we mention the publication \cite{Berdysheva}.
\end{remark}

To prove Theorem \ref{prop:AllS}, we need two auxiliary lemmas.

\begin{lemma}\label{l:Lloc}
Let $f\in L^1_{\rm loc}(G,\mathbb{R})$ be any function. Then for any function $\varphi \in \CCC_c(G,\mathbb{R})$ we have that $f\star \varphi \in \CCC(G,\mathbb{R}) \cap L^1_{\rm loc}(G,\mathbb{R})$. Furthermore, for any $V\in \BB$ and any $\varepsilon > 0$ there exists a doubly positive function $\varphi \in {\CCC_c}(G,\mathbb{R})$ such that $\|f\star \varphi - f\|_{L^1(V)}<\varepsilon$ holds. Moreover, in this last statement we can take $\ffi$ to be a constant multiple of the convolution square of the characteristic function of any sufficiently small neighborhood $W$ of $0$.
\end{lemma}

\begin{proof} This is only a slight variant of the more familiar standard case when $f \in L^1(G,\mathbb{R})$ (see, for instance, \cite[1.1.8. Theorem]{rudin:groups}).

As $f\in L^1_{\rm loc}(G,\mathbb{R})$, we also have $f\in L^1(V')$ whenever $V'\in \BB$. So let us take a fixed open $V'$ with compact closure but with $\overline{V}\Subset V'$. Then according to the known case of summable functions applied to $g:=f|_{V'}$ there exists a neighborhood $U$ of the origin such that whenever $u$ is a nonnegative weight function with support in $U$ and total mass $\int_G u=1$, then $\|g-g\star u\|_1 <\ve $.

Now we want to construct a function $u$ which is a positive definite convolution square. To this end, we take a sufficiently small open neighborhood $W$ of $0$ (specified later) and define $\ffi:=( \chi_W \star \widetilde{\chi_{W}})/\la(W)$.
As it was noted in Example~\ref{E:convsquare}, the convolution square of an $L_2(G,\mathbb{R})$-function is positive definite, whence $\ffi \in \DC$. Moreover, it is easy to see that
$$
\ffi(x)=\frac{1}{\la(W)} \int \chi_W(x-y) \widetilde{\chi_W}(y) dy = 0
$$
unless $x \in  W-W$. Also, by construction $\ffi \ge 0$ and $\int_G \ffi =1$. Therefore, $\ffi$ makes an admissible weight function $u$ provided that $W-W \subseteq U$ which is one condition on $W$ to be satisfied when choosing it. Such open neighborhoods of $0$ exist by the continuity of the group operation $(x,y) \mapsto x-y$. So, if $W_0$ is one such neighborhood, then we will restrict ourselves to consider $W\subseteq W_0$. Continuity of $f\star \varphi$ follows from (R1).

Next we aim to ascertain that on $V$ replacing $f$ by $g:=f|_{V'}$ does not change anything, more precisely, we have
$f(x)=g(x)$ and $f\star \ffi(x)=g\star \ffi(x)$ for $x\in V$.
Clearly, we have $f=f|_{V'}$ when $V\subseteq V'$. One also has to note that $$f\star \ffi(x) =\int _G f(x-y)\ffi(y) dy = \int_{W-W} f(x-y) \ffi(y) dy$$ for $\ffi(y)=0$ if $y \not\in W-W$. It follows that in case $x-W+W \subseteq V'$ we have
$$
f\star \ffi (x) = \int_{W-W} f(x-y) \ffi(y) dy = \int_{W-W} f|_{V'}(x-y) \ffi(y) dy = g \star \ffi (x).
$$
We want this for all $x\in V$, whence our second requirement for $W$ is that $V+W-W \subseteq V'$.

It remains to see the standard fact that this is indeed satisfied for some open neighborhood $W'$ of $0$.
As $V'$ is open, certainly for any $x\in V'$ the inclusion
\begin{equation}\label{eq:xZZZZ}
x+Z+Z-Z-Z \subseteq V'
\end{equation}
holds true with some open neighborhood $Z=Z_x$ of $0$, by the continuity of the group operation. Further the sets $\{ x+Z-Z ~:~ x\in \overline{V}\}$ obviously form a cover of $\overline{V}$, that is,
$$
\overline{V} \subseteq \bigcup_{x\in\overline{V}} (x+Z_x-Z_x).
$$
By compactness, there exists a finite subcover satisfying
$$
\overline{V} \subseteq \bigcup_{j=1}^n (x_j+Z_{x_j}-Z_{x_j}).
$$
Taking $W':=\bigcap_{j=1}^n Z_{x_j}$ we finally find for any $x\in V$ that with some $j$ it holds $x\in (x_j+Z_{x_j} - Z_{x_j} )$ and thus in view of \eqref{eq:xZZZZ}
$$
x + W' - W' \subseteq (x_j+Z_{x_j}-Z_{x_j})+Z_{x_j} - Z_{x_j} \subseteq V'.
$$
If $W$ is taken as any open neighborhood contained in $W_0\cap W'$, then all the above requirements are fulfilled. Moreover, $\| f- f\star \ffi\|_{L^1(V)} = \| g- g\star \ffi\|_{L^1(V)} <\ve$ holds, as desired.
\end{proof}

The following auxiliary statement is certainly folklore, however, our references do not contain a proof for it.

\begin{lemma}\label{l:convposdef}
If $f\in \D^{\star}$ and $\varphi \in \DC$, then $f\star \varphi \in \D$.
\end{lemma}

\begin{proof}
Take a compactly supported Boas-Kac root $g\in L^2(G,\mathbb{R})$, provided by Lemma~\ref{l:BoasKac}, for which $\varphi=g\star \widetilde{g}$. Thus, for any weight function $u\in\CCC_c(G,\mathbb{R})$ we find
\[
u\star \widetilde{u} \star f \star \varphi (0)=u\star \widetilde{u} \star f \star g\star \widetilde{g}(0)=u\star g \star \widetilde{u} \star \widetilde{g} \star f (0) = (u\star g) \star \widetilde{u \star g} \star f (0) \geq 0
\]
where the last inequality follows from the facts that $f \in \D^{\star}$ and $u\star g \in \CCC_c(G,\mathbb{R})$ in view of (R4).
This means that $f \star \varphi$ is an integrally positive definite function, while continuity of $f\star \varphi$ follows from (R1).
Hence in virtue of \cite[Proposition 4]{God} the function $f \star \varphi$ is positive definite as well.
\end{proof}

The following technical lemma is in fact an easy consequence of the later presented version of the Gelfand-Raikov Theorem (see Theorem~\ref{th:GelfandRaikov}), however, the lemma can be proved directly in a quite elementary fashion as well.

\begin{lemma}\label{l:nonzero}
If $0 \le f\in \D^{\star}$ and $f$ is not locally almost everywhere zero, then we necessarily have $\int_{O} f > 0$ for all open $0 \in O \in \mathcal{B}_0$.
\end{lemma}

\begin{proof}
Assume for a contradiction that there is an $0 \in O \in  \mathcal{B}_0$ such that $\int_{O} f = 0$. 
Take a $\varphi \in \D_c$ supported in $O$.
It follows directly from Lemma~\ref{l:convposdef} that $f\star \varphi \in \D$.
As a result
\[
\|f \star \varphi \|_{\infty} = (f \star \varphi)(0)=\int_{O}f(x)\varphi(-x)\leq \|\varphi\|_{\infty} \int_O f =0
\]
which yields $f \star \varphi = 0$ whenever $\varphi \in \D_c$, supported in $O$, and so in particular when $\varphi$ is a constant multiple of the convolution square of a characteristic function of a small enough neighborhood of $0$.
So it follows from Lemma~\ref{l:Lloc} that for all $\varepsilon >0$ and for any open $V\in \BB$ the estimate  $\|f\|_{L^1(V)}= \|f-f\star \varphi \|_{L^1(V)} < \varepsilon $ holds, that is, $f \equiv 0$ a.e. on $V$, whence locally a.e. on $G$.
\end{proof}

Now we are in a position to present the proof of the main result of the section.

\begin{proof}[Proof of Theorem \ref{prop:AllS}] Let $U,V\in \BB$ be given.
We divide the proof into two parts, the first being $S(U,V)=S^\star(U,V)$ while the second is $ S(U,V)=S_c(U,V)$.

First we intend to show that for any $0\le f \in \DL$ with $f\not\equiv {0}$ on the given $U$ (which by Lemma~\ref{l:nonzero} entails $\int_U f > 0 $), and any $\ve>0$ there exists another function $0\le f_1 \in \D$ with the property $\|f-f_1\|_{L^1(U \cup V)} <\ve $.
If having proved this, we get
\[
\frac{\int_V f_1}{\int_U f_1} \ge \frac{\int_V f - \ve}{\int_U f+\ve}
\]
yielding that the supremum on the left, that is, $S(U,V)$ is at least as large as the quantity $$\frac{\int_V f - \ve}{\int_U f+\ve}.$$ Taking supremum on the right (with respect to $\ve >0$ first) implies that $S(U,V) \ge S^\star(U,V)$. The other direction $S(U,V)\le S^\star(U,V)$ being obvious we will thus infer $S(U,V) = S^\star(U,V)$.

For the proof of the existence of such a function $f_1$, we apply Lemma \ref{l:Lloc} to the set $U\cup V$. Then with an appropriate doubly positive $\ffi \in \D_c$ we get $\|f-f\star\ffi\|_{L^1(U\cup V)} <\ve$. An easy application of Lemma~\ref{l:convposdef} yields $f_1:=f\star \varphi \in \D$. Since we have $\ffi\ge 0$, obviously $f\ge 0$ entails $f_1\ge 0$. This concludes the proof of the first part.

To start the second part, now we prove that any $0\le f_1\in \D$ can be approximated uniformly on the compact set $K:=\overline{U\cup V}$ by a nonnegative $f_2\in \DC$ with an error as small as it is desired. For this we just need a continuous positive definite compactly supported function $k\ge 0$ which approximates the constant 1 function ${\bf 1}$ uniformly within an arbitrarily fixed error $\ve>0$ on $K$. Such a function, even lying strictly above {\bf 1} on $K$, is provided by Lemma \ref{l:Kg} because $k\gg 0$ implies 
$$
k(z)\le k(0) \le 1+\ve
$$ 
for all $z \in G$.
We can then take $f_2:=f_1\cdot k$, which is again compactly supported (as its support is a subset of $\supp k$), continuous and also positive definite, being the product of two positive definite functions.
It follows that 
$$
\|f_1-f_1\cdot k\|_{L^\infty(K)} \le \|1-k\|_{L^\infty(K)} \cdot \|f_1\|_{L^\infty(K)} \le \ve \cdot \|f_1\|_{L^\infty(K)}
$$ 
and the function $f_2\geq 0$ provides arbitrarily good uniform approximation to $f_1$ on $K$.

To conclude the proof of the second part, let us choose any $\eta>0$ and an approximant $0\le f_2\in \DC$ satisfying $\|f_1-f_2\|_{L^\infty(K)} <\eta $. With this we find $$\int_K |f_1-f_2| \le \eta \cdot \la(K).$$ As $\la(K)<\infty$ and $\eta>0$ is arbitrarily small, we get an arbitrarily close $L^1(U\cup V)$ approximation as above implying $S(U,V)=S_c(U,V)$, as before.

Combining this with the first part $S(U,V)=S^{\star}(U,V)$ yields the full statement.
\end{proof}

\section{Proof of the main theorem}\label{sec:Halasz}

In this section we return to our original problem, and find conditions to the inequality
\begin{equation} \label{eq:int_est}
\int_G f d\nu \leq C \int_G f d\mu
\end{equation}
to hold for all continuous doubly positive functions $f$ with compact support. Here we assume that the occurring Radon measures are finite: $\mu, \nu \in {M}(G)$, or, in other words $\mu, \nu$ belong to the dual of ${\CCC_0}(G,\mathbb{R})$.

We prove the following slightly more general version of Theorem~\ref{th:Halasz}.

\begin{theorem}\label{th:summary}
Let $\mu, \nu \in {M}(G)$ be arbitrary and $C\in \RR$ be any constant. Then the following statements are equivalent:
\begin{itemize}
    \item[(i)] The inequality \eqref{eq:int_est} is satisfied for all $f\in \DC \cap\PP$.
    \item[(ii)] The inequality \eqref{eq:int_est} is satisfied for all $f\in \DD \cap\PP$.
    \item[(iii)] The inequality \eqref{eq:int_est} is satisfied for all $f\in \D \cap\PP$.
    \item[(iv)] $C \mu - \nu \in \mathcal{P}_0^{+} + \DD^{+} = M_+(G) + \MM + \mathcal{O}$.
\end{itemize}
\end{theorem}

\begin{proof}

The implications (iii) $\Longrightarrow$ (ii) $\Longrightarrow$ (i) are immediate, so we first verify (i) $\Longrightarrow$ (iii).

To this end, assume that \eqref{eq:int_est} holds for all $f \in \D_c \cap \PP$. If $\eta>0$ is chosen and $K\Subset G$ is a sufficiently large compact set, then $|\mu|(G\setminus K) <\eta$ and $|\nu|(G\setminus K)<\eta$. Hence
$$
\Bigl|\int_{G\setminus K} f d[C\mu-\nu]\Bigr| < \|f\|_\infty (|C|+1) \eta <\ve
$$
for an appropriately chosen small $\eta$.
So choose $k \in \DC$ to be an approximation of the constant 1 function ${\bf 1}$ satisfying $1\le k \le 1+\de$ on $K$ and of course $0\le k \le 1+\de$ on $G$ everywhere. Such a function exists for any choice of $\de>0$ in view of Lemma \ref{l:Kg}. With this function we find that
\[
\Bigl|\int_K (f -f\cdot k) d[C\mu-\nu] \Bigr| \le \de \|f\|_\infty \cdot \left(|C| \|\mu\| + \|\nu\|\right)
<\ve
\]
whenever $\de$ is small enough. On the other hand
\[
\Bigl|\int_{G\setminus K} f\cdot(1-k) d[C\mu-\nu]\Bigr| \le \|f\|_\infty \cdot \eta \cdot (|C|+1),
\]
as above, whence this part stays below $\ve$. Combining these we arrive at the inequality
\[
\Bigl|\int_{G} f\cdot(1-k) d[C\mu-\nu] \Bigr|<2\ve.
\]
Hence
\[
\int_{G} f d[C\mu-\nu] \ge \int_{G} (f\cdot k) d[C\mu-\nu]-2\ve \ge -2\ve
\]
because $f\cdot k \in \PP \cap\DC$ and the inequality \eqref{eq:int_est} was assumed to hold for such functions. Finding the same with arbitrary $\ve>0$ proves that \eqref{eq:int_est} holds for all $0\le f\in \D$ as well.

As a result, up to here we have seen the equivalences (i) $\Longleftrightarrow$ (ii) $\Longleftrightarrow$ (iii).

Further (ii) means that the functional $C\mu-\nu \in {M}(G)$ takes nonnegative values on functions of the intersection of the convex cones $\PPP$ and $\DD$, that is, $C\mu-\nu \in (\PPP\cap\DD)^{+}$. Description of this dual cone has been established in Theorem~\ref{th:dualcone}, whence the equivalence (ii) $\Longleftrightarrow$ (iv).
\end{proof}

\begin{proposition} \label{P:1}
$(\mathcal{P}_0\cap \mathcal{D}_0)^{+} \cap -(\mathcal{P}_0\cap \mathcal{D}_0)^{+}=\OOO$.
\end{proposition}

\begin{proof} A bounded odd measure is orthogonal to $\mathcal{P}_0\cap \DD$, whence it belongs to both cones $\mathcal{P}_0\cap \DD$ and $-(\mathcal{P}_0\cap \DD)$. Therefore, $\OOO$ belongs to their intersection, too.

Conversely, let $\si$ be a bounded measure belonging to the intersection of these two cones. Taking $\sigma =\omega + \tau+ o$ to be any decomposition of the measure $\sigma$ as an element of  $M_+(G) + \MM + \OOO$, furnished by Theorem \ref{th:dualcone}, we get in view of $\int_G d o=0$ that
$$
\widehat{\sigma}(0)=\widehat{\omega}(0)+\widehat{\tau}(0)+\widehat{o}(0)=
\int_G d\omega + \int_G d\tau \ge 0,
$$
as for $\omega \geq 0$ we clearly have $\ft{\omega} (0)\geq 0$, and $\tau \in  \MM$ is equivalent to $\widehat{\tau}(\gamma) \ge 0  ~(\forall \gamma \in \ft{G})$, by Proposition~\ref{P:comBochner}.

In a similar fashion, we can decompose $-\sigma=\alpha+\beta+\rho$ with some $\alpha \in M_+(G)$, $\beta \in \MM$ and $\rho \in \OOO$. This yields $-\widehat{\sigma}(0)\geq 0$ which entails that $\widehat{\sigma}(0)=0$. Therefore, taking into account that both $\ft{\omega}(0) \ge 0$ and $\ft{\tau}(0)\ge 0$, we get from $0=\ft{\sigma}(0)=\ft{\omega}(0)+\ft{\tau}(0)$ that $\ft{\omega}(0)=\ft{\tau}(0)=0$. By the same way, we obtain $\ft{\alpha}(0)=\ft{\beta}(0)=0$. However, $\ft{\omega}(0)=\omega(G)$ means that $\omega = \bf{0}$ in view of our a priori knowledge of $\omega \in M_+(G)$. Similarly, we also conclude that $\alpha = \bf{0}$.

Thus $\sigma=\tau+o=-\beta-\rho$ where $\tau,\beta$ are (real) bounded measures of positive type. It follows that $\tau+\beta = -o-\rho \in \OOO$ is an odd measure, while $\tau+\beta \in \mathcal{M}$. However, from the second part of Proposition \ref{pr:realrestriction} it follows that the real valued measure $\tau+\beta$ of positive type must be even. As the even measure $\tau+\beta$ can be equal to the  odd measure $-o-\rho$ only if both sides are ${\bf 0}$, we conclude that $\tau+\beta={\bf 0}$.

To conclude the proof, it remains to see why both $\tau$ and $\beta$ must be ${\bf 0}$. Taking Fourier transforms, we find $\ft{\tau} , \ft{\beta}\ge 0$
according to Proposition~\ref{P:comBochner},
while
their sum is $\widehat{\tau}+\widehat{\beta}=\widehat{\bf 0}\equiv 0$. So, indeed, $\ft{\tau}\equiv 0, \ft{\beta}\equiv 0$. As above, by the uniqueness of Fourier transform it follows that $\tau={\bf 0}$ and $\beta={\bf 0}$.
Therefore, $\sigma= \omega+\tau+o = {\bf 0} + {\bf 0} +o$, whence $\sigma$ is an odd measure, as claimed.
\end{proof}

Next we collect certain properties of the admissible constants $C$.

\begin{proposition} \label{p:Amn}
For arbitrary $\mu, \nu \in {M}(G)$ the set
\[
A(\mu,\nu):=\{ C~:~ C \mu - \nu \in (\mathcal{P}_0\cap \DD)^{+} \}
\]
possesses the following properties.
\begin{itemize}
    \item[(i)] $A(\mu,\nu)$ is a closed subinterval of $\mathbb{R}$;
    \item[(ii)] $0\in A(\mu,\nu)\iff \int_G f d\nu \leq 0\quad (\forall f \in \mathcal{P}_0\cap \DD)\iff \nu \in -(\mathcal{P}_0\cap \DD)^{+}$;
    \item[(iii)] $A(\mu,-\nu)= A(-\mu,\nu) = -A(\mu,\nu)$.
    \item[(iv)] If $A(\mu,\nu)\neq \emptyset$, then we have $ \mu \in (\mathcal{P}_0\cap \DD)^{+}$ $\iff$  $\sup A(\mu,\nu)=+\infty$;
    \item[(v)] $A(\mu,\nu)=\mathbb{R}$ $\iff$ $\nu \in -(\mathcal{P}_0\cap \DD)^{+}$ and $\mu \in \OOO$;
    \item[(vi)] If $\mu \not\in \mathcal{P}_0^{+} + \DD^{+}$ and $ \mu \not\in -(\mathcal{P}_0\cap \DD)^{+}$, then $A(\mu,\nu)$ is bounded.
\end{itemize}
\end{proposition}

\begin{proof}
Properties (i)--(iii) are consequences of the very definition of $A(\mu,\nu)$, and the description of the dual cone $(\mathcal{P}_0\cap \DD)^{+}$ in Theorem~\ref{th:dualcone}).

As for (iv), assume first $A(\mu,\nu)\neq \emptyset$ and $\sup A(\mu,\nu)=\infty$, and consider with any $C>0$ and $C\in A(\mu,\nu)$ the inequality
\[
0\leq \int_{G} f d[C\mu-\nu] = C \int_{G} f d\mu - \int_{G} f d\nu \qquad (f\in \mathcal{P}_0\cap \DD),
\]
directly implying that for any given $f \in \PP_0 \cap \DD$
\begin{equation}\label{eq:fdmudnu}
\frac{1}{C} \int_G f d\nu \leq  \int_G f d\mu \quad (f\in \mathcal{P}_0\cap \DD).
\end{equation}
Taking $C\to \infty$ on the left hand side (which is possible for $\sup A(\mu,\nu)=\infty$ was assumed), gives us $\int_G fd\mu \ge 0$, so that $\mu \in (\mathcal{P}_0\cap \DD)^{+}$. This verifies the necessity.

For the sufficiency, we note that if $\mu \in (\mathcal{P}_0 \cap \DD)^{+}$, then so is $\kappa \mu$ with any $\kappa > 0$. Hence $C\in A(\mu,\nu)(\ne \emptyset)$ implies $C+\kappa \in A(\mu,\nu)$ from which $[C,\infty)\subset A(\mu,\nu)$ and $\sup A(\mu,\nu)=\infty$ follows.

To verify (v), assume first that $C>0$ and $C\in A(\mu,\nu)$. Then we have for any given $f\in \mathcal{P}_0\cap \DD$ the validity of \eqref{eq:fdmudnu}, and thus taking the limit $C\to +\infty$ yields again that $\int_G f d\mu \geq 0$ holds for all $f\in \mathcal{P}_0\cap \DD$. Similarly, taking negative values of $C \in A(\mu,\nu)$ into account, we conclude that $\int_G f d\mu \leq 0$ is satisfied for all $f\in \mathcal{P}_0\cap \DD$. Therefore, by Proposition~\ref{P:1}, we have $\mu\in \OOO$. Further $0 \in A(\mu,\nu)$ means $-\nu \in (\mathcal{P}_0\cap \DD)^{+}$, that is, $\nu \in -(\mathcal{P}_0\cap \DD)^{+}$. The converse is obvious.

Property (vi) follows from (iii) and (iv), noting that in case $A(\mu,\nu)=\emptyset$ we have nothing to prove.
\end{proof}

\section{The case of atomic measures}\label{sec:atomic}

When both $\mu$ and $\nu$ are atomic measures, in our problem we can make further specializations regarding the occurring measures in the representation of $C\mu-\nu$.

\begin{theorem}\label{th:atomic}
The atomic measure $\sigma=\sigma_{\rm at}$ lies in $(\mathcal{P}_0\cap \DD)^{+}$ if and only if there exist
\begin{itemize}
\item[(i)] a nonnegative atomic even measure $\omega_{\rm at}\ge 0$;
\item[(ii)] an atomic real measure of positive type $\rho_{\rm at}\in \mathcal{M}$;
\item[(iii)] an atomic real odd measure $o_{\rm at}$
\end{itemize}
such that $\sigma_{\rm at}=\omega_{\rm at}+\rho_{\rm at}+o_{\rm at}$.
\end{theorem}

We already know the existence of some decompositions. The novelty here is that all the components can be taken atomic, too.

The crux of our proof will be the following possibly well-known fact, which we could not find a reference for, although we find it interesting also in its own right.

\begin{lemma}\label{l:atomicpartpostype} Assume that the measure $\tau \in M(G)$ is of positive type, that is, $\tau \in \MM$ (or is of positive type in the real sense, that is, $\tau \in \NNN$). Then so is the atomic part $\tat$ of $\tau$, too.
\end{lemma}

Before presenting the proof of the lemma, let us see how it leads to the conclusion of Theorem~\ref{th:atomic}.

\begin{proof}[Proof of Theorem~\ref{th:atomic}]
Let $\s=\om+\tau$ be any decomposition of $\s$ as an element of $(\mathcal{P}_0 \cap \DD)^{+} = M_{+}(G)+\NNN$. 
Let us decompose the occurring measures to their atomic and continuous parts as
$$
\s=\sat,\quad \omega=\oat+\om_{\rm c},\quad \tau=\tat+\tau_{\rm c}.
$$
Apparently then $\sigma=\sigma_{\rm at}=\omega_{\rm at}+\tau_{\rm at}$.
Being the atomic part of a nonnegative measure, $\om_{\rm at} \ge 0$. For the atomic part of $\tau$, we can apply Lemma~\ref{l:atomicpartpostype}: as $\tau \in \NNN$, we find $\tat \in \NNN$ as well.

To conclude the proof, we refer to the decomposition $\NNN=\MM+\OOO$, see Remark~\ref{R:MNO}. Let $\tat=\rho_{\rm at} + o_{\rm at}$ where $\rho_{\rm at}:=(\tat+\widetilde{\tat})/2$ and $o_{\rm at}:=(\tat-\widetilde{\tat})/2$ are the unique even and odd parts of $\tat$, respectively, both remaining atomic together with $\tat$.
Then $\s=\s_{\rm at}=\omega_{\rm at}+\rho_{\rm at} + o_{\rm at}$
where $\omega_{\rm at} \geq 0$, $o_{\rm at} \in \OOO$ and $\rho_{\rm at} \in \NNN$ by construction, moreover $\rho_{\rm at}$ being even, yet we have $\rho_{\rm at} \in \MM$.
\end{proof}

The forthcoming discussion will eventually lead to the proof of Lemma \ref{l:atomicpartpostype}. However, for that we need some preliminaries on almost periodic functions and mean value functionals.

Let $f$ be a function on $G$. For any $g\in G$ denote by $\mathcal{T}_g$ the $g$-translate of $f$ defined by $\mathcal{T}_g f(x):=f(g+x)$. The function $f$ is said to be \textit{almost periodic} if for all $\ve>0$ there exists a finite set $\{g_j\,:~ j=1,\dots,n\}$ such that the translates $\{\mathcal{T}_{g_j} f   ~:~ j=1,\dots,n\}$ constitute an $\varepsilon$-net among all the translates $\{\mathcal{T}_{g} f   ~:~ g\in G\}$ of $f$ in the uniform norm metric.

We will use the fact that a translation invariant mean value functional $\MEAN:=\MEANG$ exists on the set of all almost periodic functions on any locally compact group $G$ \cite[(18.8) Theorem]{HewittRossI}. Hence also on $\ft{G}$ which will be denoted by $\MEANGH$. Note that this mean value operation is a bounded positive linear functional on the set of all almost periodic functions and is normalized so that $\MEAN({\bf 1})=1$. Furthermore, this mean value functional is unique \cite[(18.9) Theorem]{HewittRossI}. Also, it vanishes on all characters $\chi$ not identically one (for any $z\in G$ with $\chi(z)\ne 1$ translation invariance entails $\MEAN(\chi(x))=\MEAN(\chi(z+x))=\chi(z)\MEAN(\chi(x)$) and thus $(1-\chi(z))\MEAN(\chi)=0$).

%The mean value functional extends to a much larger class than the set of all almost periodic functions, at least to the class of so-called weakly almost periodic functions which simultaneously contains the following classes of functions.
%\begin{itemize}
%\item[a)] Almost periodic functions.
%\item[b)] Positive definite functions.
%\item[c)] Every $\CCC_0(G,\mathbb{C})$ functions.
%\end{itemize}
%The relevant reference in this direction is the classical work of Eberlein \cite{Eberlein1}.

Next we recall some results of Eberlein \cite{Eberlein1} which will be needed.
Let $\s:=\sat$ be an atomic bounded Radon measure. Then $\sat$ is of the form $\sat=\sum_{j=1}^{\infty} a_j \delta_{x_j}$ (with $x_j$ running over a countable subset in $G$, and $\de_x$ denoting the Dirac measure concentrated at $x$), so that its Fourier transform is
\[
\widehat{\s}(\gamma)=\widehat{\s_{\rm at}}(\gamma)=\int_{{G}} \overline{\gamma}(x) d {\sat}(x)=\sum_{j=1}^{\infty} a_j \overline{\gamma}(x_j).
\]
Note that $\|\s_{\rm at}\|=\sum_{j=1}^\infty |a_j| <\infty$ implies that the series expansion of $\widehat{\s_{\rm at}}$ is normally (and thus absolutely and uniformly) convergent and thus defines a continuous function. Therefore, $\widehat{\s_{\rm at}}\in \CCC_b(\ft{G},\mathbb{C})$. Furthermore, it is necessarily an almost periodic function, see \cite[(18.3) Theorem (iv)]{HewittRossI}. This statement can be reversed due to the following result of Eberlein \cite[Theorem 3]{Eberlein2}.

\begin{lemma}[Eberlein]\label{l:atomicap} The Fourier transform of a bounded Radon measure is almost periodic if and only if the measure itself is a bounded atomic measure.
\end{lemma}

We postpone the proof of the nontrivial part of the lemma, continuing the discussion first.
We have seen that the atomic masses determine a convergent series representation of the Fourier transform. In what follows, we explain that conversely, the mass of the atomic component $\sigma(\{x_0\})=\sigma_{\rm at}(\{x_0\})$ of $\sigma$ at the arbitrary but fixed point $x_0\in G$ (and thus the whole atomic measure itself) can be reconstructed using the mean value functional as
\begin{equation}\label{eq:atomicbymeanvalue}
\s_{\rm at}(\{x_0\})=\sigma(\{x_0\})=\MEANGH \left( \gamma(x_0)\widehat{\s}(\gamma) \right).
\end{equation}
First for any $\s \in M(G)$ the mean value functional $\MEANGH$ can be applied to $\gamma(x_0)\widehat{\s}(\gamma)$, as decomposing $\s=\s_{+}-\s_{-}$ with both $\s_{+}, \s_{-}\in M_+(G)$, we find that $\ft{\s_{+}},~ \ft{\s_{-}}$, together with their products with the positive definite character $\gamma \mapsto\gamma(x_0)$ on $\ft{G}$, are continuous positive definite functions on $\ft{G}$, to which $\MEANGH$ is certainly extended.

Recall that a measure $\mu$ is called \emph{continuous}, if for all singletons $\si(\{x_0\})=0$. Clearly, taking $\sat:=\sum_{x_j} \sigma(\{x_j\}) \delta_{x_j}$ with all the points $x_j$ having nonzero mass is an atomic measure, and the left over remainder $\s-\sat$ is a continuous measure. It remains to establish that the measure of a singleton $\{x_0\}$ equals to the mean value in \eqref{eq:atomicbymeanvalue}. However, let us point out that Eberlein has proved the following even stronger result (cf. \cite[Theorem 15.2]{Eberlein1} and \cite[Theorem 1]{Eberlein2}), as well.

\begin{lemma}[Eberlein]\label{l:meanvalueParseval} If the bounded Radon measure $\mu\in M(G)$ has the decomposition $\mu=\mu_{\rm at}+\nu$, where $\mu_{\rm at}=\sum_{j=1}^\infty a_j \de_{x_j}$ is the atomic component of $\mu$ and $\nu$ is the continuous component of $\mu$ then $\MEANGH(|\widehat{\mu}|^2)=\MEANGH(|\widehat{\mu_{\rm at}})|^2)=\sum_{j=1}^\infty |a_j|^2$ and $\MEANGH(|\widehat{\nu}|^2)=0$.
\end{lemma}

From here \eqref{eq:atomicbymeanvalue} follows immediately using the "orthogonality relations" (that is, $\MEAN({\bf 1})=1$ and $\MEAN(\chi)=0$ for any character not identically one) and standard manipulations with Parseval formula.
We will now discuss how Lemma \ref{l:meanvalueParseval} entails Lemma \ref{l:atomicap}, more precisely, its nontrivial part stating that $\mu$ is atomic provided its Fourier transform is almost periodic.

\begin{proof}[Proof of Lemma~\ref{l:atomicap}]
As a special case of Lemma \ref{l:meanvalueParseval}, we already know that the mean square value of the Fourier transform of a bounded Radon measure vanishes if and only if the measure is continuous. Consider an arbitrary measure $\mu$ with almost periodic Fourier transform $\ft{\mu}$. Decomposing as in the Lemma, the Fourier transform $\ft{\mu_{\rm at}}$ of the atomic component $\mu_{\rm at}$ is of course almost periodic, whence so is the difference $\widehat{\mu}-\widehat{\mu_{\rm at}}=\ft{\nu}$ -- the Fourier transform of the continuous part of $\mu$ --, too. Recall that if the mean square value \emph{of an almost periodic function} is zero, then the function itself is identically zero (see e.g. \cite{Eberlein1} or \cite[(18.8) Theorem (i), (ii)]{HewittRossI}). Thus, Lemma~\ref{l:meanvalueParseval} entails that in this case $\nu$ is identically zero. Hence $\mu=\mu_{\rm at}$ is atomic.
\end{proof}

In the proof of Lemma~\ref{l:atomicpartpostype}, we also need the following useful observation.

\begin{lemma}\label{l:posappartpos} If $0\le \phi \in \CCC_b(G,\mathbb{R})$ and $\phi=\phi_{\rm ap} + \phi_{\rm s}$ is a decomposition of $\phi$ to an almost periodic part $\phi_{\rm ap}$ and a complementing small part $\phi_{\rm s}$ satisfying $\MEAN(|\phi_{\rm s}|^2)=0$, then we necessarily have $\phi_{\rm ap}\ge 0$.
\end{lemma}

Note that in the conditions of the Lemma we do not assume that $\phi \gg 0$ (which need not be true) but we assume the very existence of the above decomposition and nonnegativity of $\phi$.

\begin{proof} For any $w\in \RR$ define $w_{+}:=\max(w,0)$ and $w_{-}:=\min(w,0)$. Now if $a\ge 0$ and $b\in \RR$, then $|(a-b)_{-}|\le b_{+}$ holds true. Using this, we can write
$$
0\le \MEAN(|(\phi_{\rm ap})_{-}|^2) = \MEAN(|(\phi-\phi_s)_{-}|^2) \le \MEAN(|(\phi_s)_{+}|^2) \le \MEAN(|\phi_s|^2)=0,
$$
whence the function $(\phi_{\rm ap})_{-}$ has zero mean square, that is, the nonnegative function $|(\phi_{\rm ap})_{-}|^2$ has zero mean.

Note that once $f$ is an almost periodic function on $G$, then so is its negative part $f_{-}$. Indeed, we have for any two translates and any point $x\in G$ the inequality
$$
|\mathcal{T}_g f_{-}(x)-\mathcal{T}_hf_{-}(x)|\le |\mathcal{T}_gf(x)-\mathcal{T}_hf(x)|,
$$
because for any two real numbers $u, v \in \RR$ we have $|u_{-}-v_{-}|\le |u-v|$. This furnishes that  even the corresponding uniform norms satisfy
$$
\|\mathcal{T}_gf_{-}-\mathcal{T}_hf_{-}\|_{\infty}\le \|\mathcal{T}_gf-\mathcal{T}_hf\|_{\infty}.
$$
Therefore, if the translates of $f$ by a set $\{g_j\,:~ j=1,\dots,n\}$ constitute an $\ve$-net of all the translates of $f$, we necessarily have that the translates of $f_{-}$ by the same set $\{g_j\,:~ j=1,\dots,n\}$ form an $\ve$-net for the set of all the translates of $f_{-}$. Thus existence of a finite subset of translates constituting an $\ve$-net for any given $\ve>0$ for $f$ entails the same property even for the translates of  $f_{-}$ . This is equivalent to almost periodicity.

Continuity of $f$ is trivially inherited by taking $f_{-}$. So we get that $(\phi_{\rm ap})_{-}$ is a continuous almost periodic function, whence so is $|(\phi_{\rm ap})_{-}|^2$. We found that $|(\phi_{\rm ap})_{-}|^2$ is a continuous nonnegative almost periodic function on $G$ having zero mean value. As it was noted above, this implies that it is identically zero, in other words, $\phi_{\rm ap} = (\phi_{\rm ap})_{+}\ge 0$.

\end{proof}

\begin{proof}[Proof of Lemma \ref{l:atomicpartpostype}] We prove only the "real sense" version, for we need only this part in the above proof of Theorem \ref{th:atomic}. However, note that the "complex version" follows from this part.

Recall that according to Corollary~\ref{C:realsense} the measure $\tau$ is of positive type in the real sense if and only if $\Re \ft{\tau} \ge 0$. Let us decompose $\tau=\tat+\tau_{\rm c}$ as the sum of its atomic and continuous parts. We want to prove that $\tat $ is also a measure of positive type in the real sense. To do so, let us consider the respective Fourier transforms $\ft{\tau}=\ft{\tat} + \widehat{\tau_{\rm c}}$. These Fourier transforms are in $\CCC_b(\ft{G},\mathbb{C})$. Furthermore, $\ft{\tat}(\gamma)=\sum_{j=1}^{\infty} \tau(\{x_j \}) \overline{\gamma}(x_j) $ is almost periodic and $\MEAN(|\ft{\tau_{\rm c}}|^2)=0$, in view of Lemma \ref{l:meanvalueParseval}. As the sum of two almost periodic functions is almost periodic, so is the real part $\Re \ft{\tat} = (\ft{\tat} + \overline{\ft{\tat}})/2$ of $\ft{\tat}$. Clearly, we also have $\MEAN(|\Re\ft{\tau_{\rm c}}|^2)\le \MEAN(|\ft{\tau_{\rm c}}|^2)=0$.

We finally prove that the real part of the Fourier transform $\ft{\tat}$ is nonnegative which entails the assertion, by  Corollary~\ref{C:realsense}. What we know by construction is that $\tau$ is of positive type in the real sense, and thus $\Re \ft{\tau} \ge 0$. Hence an application of Lemma \ref{l:posappartpos} to $0 \leq \Re \widehat{\tau} = \Re \widehat{\tau_{\rm at}} + \Re \widehat{\tau_{\rm c}} $ furnishes $\Re \ft{\tat} \ge 0$.
\end{proof}

\section{The case of absolutely continuous measures}\label{sec:ac}

In this section we present another application of our main theorem.
As it was remarked in the introduction, Logan \cite{Log} found an upper estimate for $C(T)$. More precisely, he obtained
\begin{equation}\label{eq:Logan}
C\le C(T) = \frac{1}{2}\frac{([2T]+1)([2T]+2)}{[2T]+1-T}
\end{equation}
in the setting of Dirichl\'et polynomials $P(t)=\sum_k a_k e^{i\lambda_kt}$ with finite sums and different real exponents $\lambda_k$.
In \cite{BAMS} we reproved the upper estimate of Logan\footnote{In this regard we are indebted to G. Hal\'asz for calling the attention to a reference overlooked in \cite{BAMS}.}, in the setting of positive definite nonnegative functions.

While the proof of Logan relied on ad-hoc ideas, our approach to the solution was more direct. Namely, we observed that if $h$ is a positive definite, say continuous  function satisfying $$h\leq h_C:= C\chi_{[-1,1]}-\chi_{[a-T,a+T]}-\chi_{[-a-T,-a+T]},$$ then for every continuous $f \ggg 0$, we have
\[
0\leq \int_{-\infty}^{+\infty} f \cdot  h ~ dx \leq \int_{-\infty}^{+\infty} f \cdot h_C ~ dx = C \int_{-1}^{1} f dx - \int_{a-T}^{a+T} f dx - \int_{-a-T}^{-a+T} f dx
\]
meaning that the constant $C/2$ is admissible for the extremal problem $S([-1,1],[a-T,a+T])$.
At the end of the paper \cite{BAMS}, we conjectured that our approach is in principle optimal in the following sense.

\begin{conjecture}\label{dualconj}
The constant
$$\frac{1}{2} \inf \{C ~:~ \exists  h ~ \mbox{continuous positive definite} ~ \mbox{such that} ~ h \leq C\chi_{[-1,1]}-\chi_{[a-T,a+T]}-\chi_{[-a-T,-a+T]} \}$$
is the best among the admissible ones.
\end{conjecture}

In this section we make further specialization of Theorem~\ref{th:summary} when the occurring measures are both absolutely continuous with respect to a reference Haar measure. Then we are able to prove the above duality conjecture in the more general setting of LCA groups.
In \cite{BAMS} the following extremal quantities were introduced on the real line (see also \cite{GT2}).
Let $U, V\in \BB$ be arbitrary. Then we set
\[
\sigma(U,V):=\frac12 \inf C(U,V)
\]
where
\[
C(U,V):=\{C>0~:~ {\rm for}~ h_C:=C\chi_U-\chi_V-\chi_{-V} | ~\exists g\le h_C ~{\rm{such~that}}~ g\in \D\}.
\]
Making use of the reformulation, suggested to us by \cite{Virosz}, this can be written equivalently as
$$ \sigma(U,V):=\inf \{C~:~ h_C \in \D+\PP_\infty \}. $$ Further it is also introduced
$$ \overline{\s}(U,V):=\sup_{g\in G} \s(U,g+V). $$
Also, whenever $\lambda U$ is defined (so for $\lambda\in \NN$ for all $G$, and for general $\la>0$ at least for $\TT^d, \RR^d$)
$$
\gamma(U,\lambda):=\s(U,\lambda U).
$$
It is easy to see that $S(U,V)\le \s(U,V)$, whence also $\overline{S}(U,V) \le \overline{\s}(U,V)$ where $$\overline{S}(U,V):=\sup_{g \in G} S(U,g+V)$$ is the correspondingly derived Shapiro-type extremal quantity.
In the last result of our paper, we prove that in fact these extremal quantities are equal.

\begin{theorem}\label{th:CK}
Let $U\in \BB$ be any symmetric neighborhood of $0$, and take an arbitrary $V\in \BB$.
We have $\s(U,V)=S(U,V)$, whence also $\overline{\s}(U,V)=\overline{S}(U,V)$. Furthermore, for any $U, V$ as above, there exists a $\s(U,V)$-extremal "weight function" $g\in \mathcal{D}$ with $g\leq h_{2 S(U,V)}$.
%%%%
%%%%\begin{itemize}
%%%%    \item[a)] We have $\s(U,V)=S(U,V)$, whence also $\overline{\s}(U,V)=\overline{S}(U,V)$.
%%%%    \item[b)] There is an extremal "weight function"  with $g\leq h_{2 S(U,V)}$ such that $g\in \mathcal{D}$.
%%%%\end{itemize}
%%%%

\end{theorem}

We remark that an analogous statement remains valid for the dilated version $\gamma(U,\lambda)$. Similarly, one can analyze the extremal quantities $Q(U,k)$, too. These coincide with $\gamma(U,\lambda)$ for convex $U$ in $\RR^d$ or $\TT^d$ and for $\la=k\in \NN$. The details are left to the reader.

Before the proof of Theorem \ref{th:CK} we also recall the celebrated Gelfand-Raikov theorem, cited here from \cite[Theoreme 3]{God} in its strongest form.

\begin{theorem}[Gelfand-Raikov]\label{th:GelfandRaikov}
Let $\mu$ be a Radon measure of positive type on $G$. Assume that there exists a neighborhood $U$ of $0$, and a real number $K\geq 0$, such that for all weight functions $0\le u \in \CCC_c(G)$ and $\supp u \Subset U$, it holds
\begin{equation}\label{eq:GRcondition}
u\star \widetilde{u} \star \mu(0)\leq K \left(\int_{G} u d\lambda \right)^2.
\end{equation}
Then $\mu$ is absolutely continuous with a continuous positive definite Radon-Nikodym derivative $\Phi$ satisfying $\|\Phi\|_\infty \le K$.
\end{theorem}

For full precision we note that the uniform norm estimate is not contained in the cited version but it follows from the continuity of $\Phi$ and $\|\Phi\|_\infty=\Phi(0)$, which is property (p2) of positive definite functions from $\D$.

Indeed, for any $\ve>0$ there exist a neighborhood $V$ of $0$ such that $\Phi|_V \ge C:=\Phi(0)-\ve$. Then for any open $W \in \BB$ such that $\overline{W-W}-\overline{W-W} \subseteq U\cap V$, and for the continuous nonnegative function $u:=\chi_W\star \widetilde{\chi_{W}}$ with symmetric support $S:=\supp u \Subset \overline{W-W}$ with %%%satisfying
$S-S \subseteq U\cap V$ we get that
\[
\begin{gathered}
K  \left(\int_{G} u d\lambda \right)^2 \ge u\star \widetilde{u} \star \mu(0)
%%%% =\int_G \left( \int_G u(-y-z)\widetilde{u}(z) d\la(z) \right) \Phi(y) d\lambda(y)
= \int_G \int_G u(-y-z)\widetilde{u}(z) \Phi(y) d\lambda(y) d\la(z) \\
 = \int_{-S} \int_{-S-S}  u(-y-z)\widetilde{u}(z) \Phi(y) d\lambda(y) d\la(z)
\\
\ge \int_{-S} \int_{-S-S} u(-y-z){u}(-z) C d\lambda(y) d\la(z)  =  C \left(\int_G u\right)^2
\end{gathered}
\]
proving $C\le K$, that is, $\Phi(0)-\ve \le K$ which is equivalent to the assertion.

\begin{proof}[Proof of Theorem \ref{th:CK}] We start with the proof of part a). Clearly, it suffices to prove the equality $S(U,V)=\s(U,V)$.

First let us check the part $S(U,V)\le \s(U,V)$. If $C \in C(U,V)$ is an admissible constant such that $h_C$ stays above a continuous positive definite function $g$, then we necessarily have for any $0\le f\in \D_c$ that
$$
0\le \int f g \le  \int f g + \int f (h_C-g) = \int f h_C = C \int_U f - \int_V f - \int_{-V} f.
$$
Using evenness of $f\in \DC$, which is provided by property (p4), $$\frac{\int_V f}{\int_U f}  \le \frac{C}{2}$$ obtains. First taking supremum over $0\le f\in \D_c$ and then taking infimum over all such constants $C$ we thus infer $S_c(U,V)\le \s(U,V)$. However, we have proved as part of Theorem \ref{prop:AllS} that $S_c(U,V)=S(U,V)$, whence we even have $S(U,V)\le \s(U,V)$.

Second, we turn to the converse inequality $S(U,V) \geq \sigma(U,V)$. So now let us take an admissible constant $c$ from the definition of $S(U,V)$, that is, a constant $c$ with the property 
$$
\int_V f \le c \int_U f \quad (\forall 0 \le f\in \D).
$$ 
In other words, $c$ is subject to the condition $\int_G f [c\mu-\nu_1] \ge 0 \, (0 \le f\in \D)$ with the absolutely continuous measures $\mu:=\la|_U =\chi_U d\lambda, \quad \nu_1:=\la|_V =\chi_V d\lambda$. Then the set of admissible values $c$ for the above is exactly $A(\mu,\nu_1)$, and according to part (i) of Proposition~\ref{p:Amn}, this set is closed, i.e. $c:=S(U,V)$ itself is an admissible constant, too. Therefore, we can work right with this value $c=S(U,V)$ from here on.

Taking into consideration that all $f\in \D$ is even, we can as well consider the symmetric measures: $\mu:=\la|_U =\chi_U d\lambda$ (for $U$ was symmetric), and $\nu_2:=\widetilde{\nu_1}:=\la|_{-V} =\chi_{-V} d\lambda$. Adding, we even find with $C:=2c=2S(U,V)$ the inequality $\int_G f [C\mu-\nu] \ge 0 \, (0 \le f\in \D)$ with $\mu:=\la|_U =\chi_U d\lambda, \quad \nu:=\nu_1+\nu_2=\la|_V + \la|_{-V} =(\chi_V + \chi_{-V} )d\lambda$, too.

So the general result in Theorem~\ref{th:summary} tells us that
this holds for all $0\le f\in \D$ if and only if
$$
\rho:=C\mu-\nu \in M_{+}(G)+\MM+\OOO.
$$
Therefore, the measure $\rho$ has the representation $\rho=\om+\tau+o$ where $\om \in M_{+}(G)$, $\tau \in \MM$ is real and of positive type, and $o\in \OOO$ is real and odd.
We have that $\rho$ is absolutely continuous with Radon-Nikodym derivative $h_C \in L^\infty_c(G,\mathbb{R})$.
Moreover, $K:=\|h_C\|_\infty\le C+2$ also holds. Clearly, if $0\le u \in \CCC_c(G,\mathbb{R})$ is any weight function, then we have
$$
u\star\widetilde{u}\star \rho (0)=\int_G\int_G u(-y)u(-z) h_C(y-z) dy dz \le K \left(\int_G u\right)^2.
$$
For the nonnegative measure $\omega$ and the nonnegative weight function $u$ we obviously have $u \star \widetilde{u}\star\omega (0) \ge 0$. For the odd measure $o$, however, we have $u \star \widetilde{u}\star o (0) = 0$, because the convolution square $u\star\widetilde{u}$ is positive definite and real, whence even. The above yields
$$
u \star \widetilde{u}\star\tau (0) \le u \star \widetilde{u}\star\rho (0) \le K \left(\int_G u\right)^2.
$$
Here, $\tau\in\MM$ is a measure of positive type, whence the Gelfand-Raikov Theorem applies. This furnishes that even $\tau$ is absolutely continuous with Radon-Nikodym derivative $\phi \in \D$ and $\|\phi\|_\infty \le K$. As a result, $\om+o=\rho-\tau$ is absolutely continuous, too.

Observe further that both $\rho:=h_Cd\lambda$ and $\tau \in \MM$ was even, whence so must be the difference $\rho-\tau$. Now the representation $\omega+o$ is not necessarily unique, so we cannot state just for any representation in this form that $o={\bf 0}$, but at this point we can rewrite this side as $\omega+o =  [(\omega+o)+(\omega+o)^*]/2 + [(\omega+o)-(\omega+o)^*]/2$, where the first, even part is $\omega_1:= [(\omega+o)+(\omega+o)^*]/2= [\omega+\omega^*]/2 +  [o+o^*]/2 =  [\omega+\omega^*]/2$ because $o$ was odd. At this point, however, we can indeed observe that $\omega_1$, together with $\omega$, is nonnegative, while it is also even. The other, odd component is $o_1:= [(\omega+o)-(\omega+o)^*]/2 = [\omega-\omega^*]/2 + o$.

Recalling that $\omega+o=\omega_1+o_1$ equals to the even measure $\rho-\tau$, it follows that $o_1=\rho-\tau-\omega_1$ is both even and odd, whence is null, and we finally find that $\omega_1=\rho-\tau$ with both sides being even, and as a result of the absolute continuity of the right hand side, the same holds for $\omega_1$, too. So its Radon-Nikodym derivative $\psi$ is even and it satisfies
$$
0\le \psi \in L^\infty(G,\mathbb{R}) \cap L^1(G,\mathbb{R}).
$$
Thus, we find that
\begin{align} \label{eq:HC}
h_C=\psi + \phi \quad \mbox{with some} \quad \psi \ge 0,~ \phi \in \D
\end{align}
which is exactly the same representation that appears in the definition of the quantity $\s(U,V)$. It means that $c:=S(U,V)$, $C:=2c$ was found to be an admissible value for the extremal problem $\s(U,V)$, too. Whence $S(U,V) \ge \s(U,V)$. Since the converse inequality has been already shown, the proof of $S(U,V) = \s(U,V)$ is complete. Furthermore, \eqref{eq:HC} shows the existence of the extremal weight function $g:=\phi$ as well.
\end{proof}


\begin{thebibliography}{20}

\bibitem{Berdysheva} {\sc E. Berdysheva, Sz.Gy. R\'ev\'esz}, Delsarte's Extremal Problem and Packing on Locally Compact Abelian Groups, arXiv:1904.03614.

\bibitem{boas}{\sc R.P. Boas, Jr.}, \emph{Entire functions}, Academic Press, New York, 1954.

\bibitem{BR} {\sc A. Bonami, Sz.Gy. R\'ev\'esz}, Failure of Wiener's property for positive definite periodic functions.  \emph{C. R. Math. Acad. Sci. Paris}, {\bf 346} (2008), no. 1-2, 39--44.

%\bibitem{Cooper} {\sc J.L.B. Cooper} Positive definite functions of a real variable, \emph{Proc. London Math. Soc.}, {\bf 3(1)} (1960), 53--66.

\bibitem{Eberlein1}{\sc W.F. Eberlein}, Abstract ergodic theorems and weak almost periodic functions, \emph{Trans. Amer. Math. Soc.} {\bf 67} (1949), 217--240.

\bibitem{Eberlein2}{\sc W.F. Eberlein}, A note on Fourier-Stieltjes transforms, \emph{Proc. Amer. Math. Soc.} {\bf 6} (1955), 310--312.

\bibitem{Edwards} {\sc R.E. Edwards}, \emph{Functional Analysis: Theory and Applications}, Dover Publications Inc., New York, 1965.

\bibitem{BAMS} {\sc A. Efimov, M. Ga\'al, Sz.Gy. R\'ev\'esz}, On integral estimates of nonnegative positive definite functions, {\it Bull. Aust. Math. Soc.} {\bf 96} (2017) no. 1, 117--125.

\bibitem{EF}{\sc P. Erd\H{o}s, W.H.J. Fuchs}, On a problem of additive number theory, {\it J. London Math. Soc.} {\bf 31} (1956), 67--73.

%\bibitem{Folland} {\sc G.B. Folland}, \emph{A Course in  Abstract Harmonic Analysis}, Studies in Advanced Mathematics, CRC Press, Boca Raton FL, 1995.

\bibitem{God} {\sc R. Godement}, Les Fonctions De Type Positif et la Theorie Des Groupes, \emph{Trans. Amer. Math. Soc.} {\bf 63}, No. 1 (Jan., 1948), 1--84.

\bibitem{GT} {\sc D.V. Gorbachev, S.Y. Tikhonov}, Wiener's problem for positive definite functions,  \emph{Math. Z. } {\bf 289} (2018), No. 3--4, 859--874.

\bibitem{GT2} {\sc D.V. Gorbachev, S.Y. Tikhonov}, Doubling condition at the origin for
non-negative positive definite functions, \emph{Proc. Amer. Math. Soc.} {\bf 147} (2019), 609--618.

\bibitem{Halasz} {\sc G. Hal\'asz}, \"Uber die Mittelwerte multiplikativer zahlentheoretischer Funktionen, \emph{Acta Math. Hung.} {\bf 19} No. 3--4 (1968), 365--403.

\bibitem{Halmos} {\sc P.R. Halmos}, \emph{Finite dimensional vector spaces}, 2nd edition, Van Nostrand Inc., Princeton, Toronto, London, New York, 1958.

\bibitem{HewittRossI} {\sc E. Hewitt, K.A. Ross}, \emph{Abstract harmonic analysis}, {\bf I}, Die Grundlehren der mathemtischen Wissenchaften in Einzeldarstellungen, Band {\bf 115}, Springer Verlag, Berlin, G\"ottingen, Heidelberg, 1963.

\bibitem{HewittRossII} {\sc E. Hewitt, K.A. Ross}, \emph{Abstract harmonic analysis}, {\bf II}, Die Grundlehren der mathemtischen Wissenchaften, Band {\bf 152}, Springer Verlag, Berlin, Heidelberg, New York, Budapest, 1970.

\bibitem{GeomAnal}{\sc R.B. Holmes}, \emph{Geometric Functional Analysis and its Applications}, Springer, Berlin, 1975.

\bibitem{JMR} {\sc P. Jamming, M. Matolcsi, Sz.Gy. R\'evesz} On the extremal rays of the cone of positive,
positive definite functions, {\it J. Fourier Anal. Appl.}, {\bf 15} (2009), No. 4, 561--582.

\bibitem{jeya-wolk} {\sc V. Jeyakumar, H. Wolkowicz}, Generalizations of Slater's constraint qualification for infinite convex programs, {\it Math. Prog.}, {\bf 57} (1992), 85--101.

%\bibitem{kolountzakis:groups} {\sc M.N. Kolountzakis and Sz.Gy. R\' ev\' esz}, Tur\'an's extremal problem for positive definite functions on groups, \emph{J. London Math. Soc.}, \textbf{74} (2006), 475--496.

\bibitem{Log} {\sc B.F. Logan}, An interference problem for exponentials \emph{Michigan Math. J.}, {\bf 35} (1988), 369--393.

\bibitem{Mathias} {\sc M. Mathias}, \"Uber positive Fourier-Integrale. {\em Math. Z.} {\bf 16} (1923), 103--125.

\bibitem{ML} {\sc H.L. Montgomery}, Ten lectures on the interface between harmonic analysis and analytic number theory, CBMS Regional Conference Series in Mathematics, 84. American Mathematical Society, Providence, RI, 1994. xiv+220 pp.

\bibitem{Po} {\sc A. Povzner}, \"Uber positive Funktionen auf einer Abelschen Gruppe, \emph{Dokl. Adad. Nauk SSSR} {\bf 28} (1940) 294--295.

\bibitem{Ra} {\sc D.A. Raikov}, Positive definite functions on commutative groups with an invariant measure, \emph{Dokl. Adad. Nauk SSSR} {\bf 28} (1940) 296--300.

\bibitem{rudin:groups} {\sc W. Rudin}, {\em Fourier analysis on groups}, Interscience Tracts in Pure and Applied Mathematics, No. {\bf 12} Interscience Publishers (a division of John Wiley and Sons), New York-London 1962 ix+285 pp.

\bibitem{Sh}{\sc H.S. Shapiro}, Majorant problems for Fourier coefficients, \emph{Quart. J. Math. Oxford} Ser. (2) {\bf 26} (1975),  9--18.

\bibitem{Toeplitz} {\sc O. Toeplitz}, \"Uber die Fourier'sche Entwickelung positiver Funktionen, {\em Rend. Circ. Mat. Palermo}, {\bf 32} (1911), 191--192.

\bibitem{Virosz}{\sc D. Virosztek}, A short proof of a duality theorem and another application of an intersection formula on dual cones, \emph{Bull. Aust. Math. Soc.} Volume {\bf 97}, Issue 1 (2018), pp. 94--101

\bibitem{Wa}{\sc S. Wainger}, A problem of Wiener and the failure of a principle for Fourier series with positive coefficients, \emph{Proc. Amer. Math. Soc.} {\bf 20} (1969) 16--18.

\bibitem{We}{\sc A. Weil}, L'Int\'egration dans le Groupes Topologiques et ses Applications, Actualit\'es Sci. Indust., no. 869, Hermann, Paris, 1940.

\bibitem{WW}{\sc N. Wiener, A. Wintner}, On a local $L^2$-variant of Ikehara's theorem, \emph{Rev. Mat. Cuyana} {\bf 2} (1956), 53--59. See also in the \emph{Norbert Wiener: Collected Works}, {\bf II}, MIT Press, Cambridge MA, 1979.


\end{thebibliography}
\end{document}